\documentclass[10pt]{amsart}%
\usepackage[T1]{fontenc}
\usepackage{amsmath,amssymb}
\usepackage{fullpage}
\usepackage{amsthm}
\usepackage{multirow}
\usepackage{color}
\usepackage{enumerate}
\usepackage{tikz}
\usepackage{ytableau}
\usepackage{pifont}
\usepackage{amsmath}
\usepackage{graphicx}
  \usepackage[all]{xy}
\usepackage{amsfonts}
\usepackage{amssymb}
\usepackage{ytableau}%

\providecommand{\U}[1]{\protect\rule{.1in}{.1in}}
\providecommand{\U}[1]{\protect\rule{.1in}{.1in}}
\providecommand{\U}[1]{\protect\rule{.1in}{.1in}}
\providecommand{\U}[1]{\protect\rule{.1in}{.1in}}
\providecommand{\U}[1]{\protect\rule{.1in}{.1in}}
\newcommand{\ulambda}{{\boldsymbol{\lambda}}}

\newcommand{\umu}{{\boldsymbol{\mu}}}

\newcommand{\uemptyset }{{\boldsymbol{\emptyset}}}
\newtheorem{Th}{Theorem}[subsection]
\numberwithin{equation}{section}

\newtheorem{lemma}[Th]{Lemma}
\newtheorem{Cor}[Th]{Corollary}
\newtheorem{Prop}[Th]{Proposition}
\newtheorem{conj}[Th]{Conjecture}

\theoremstyle{remark}
\newtheorem{Rem}[Th]{Remark}{\rmfamily}
\theoremstyle{definition}
\newtheorem{Def}[Th]{Definition}{\rmfamily}
\newtheorem{exa}[Th]{Example}{\rmfamily}

\def\C{\mathbb{C}}
\def\Z{\mathbb{Z}}

\newcommand{\Irr}{\operatorname{Irr}}

\def\Irr{{\mathrm{Irr}}}
\def\Q{{\mathbb Q}}
\def\N{{\mathbb N}}
\def\Z{{\mathbb Z}}

\def\C{{\mathbb C}}

\newcommand\blfootnote[1]{%
  \begingroup
  \renewcommand\thefootnote{}\footnote{#1}%
  \addtocounter{footnote}{-1}%
  \endgroup
}

\thanks{Both authors were supported by the \emph{Agence Nationale de la Recherche} through the JCJC project ANR-18-CE40-0001 for the realisation of this project.}

\begin{document}

\title{Defect in cyclotomic Hecke algebras}

\author{Maria Chlouveraki}
\address{National and Kapodistrian University of Athens, Department of Mathematics, 
Panepistimioupolis 15784 Athens, Greece.}
\email{mchlouve@math.uoa.gr}

\author{Nicolas Jacon}
\address{Universit\'{e} de Reims Champagne-Ardennes, UFR Sciences exactes et
naturelles. Laboratoire de Math\'{e}matiques UMR CNRS 9008. Moulin de la Housse BP
1039. 51100 REIMS. France.}
\email{nicolas.jacon@univ-reims.fr} 

\maketitle
\date{}
\blfootnote{\textup{2020} \textit{Mathematics Subject Classification}: \textup{20C08, 05E10, 20C20}} 
\begin{abstract}

The complexity of a block of a symmetric algebra can be measured by the notion of defect, a numerical datum associated with each of the simple modules contained in the block. Geck showed that the defect is a block invariant for Iwahori--Hecke algebras of finite Coxeter groups in the equal parameter case, and speculated that a similar result should hold in the unequal parameter case. We prove that the defect is a block invariant for all cyclotomic Hecke algebras associated with the complex reflection groups of the infinite series $G(l,p,n)$, which include the Weyl groups of type $B_n$ in the unequal parameter case. In particular, for the groups $G(l,1,n)$, we show that the  defect corresponds to the notion of  weight in the sense of Fayers.   We thus also obtain a new way of computing the weight, which uses a generalisation of the notion of hook lengths. We further show computationally that the defect is a block invariant for all cyclotomic Hecke algebras of exceptional type for which the blocks are known, and we conjecture that the result should hold for all complex reflection groups. Finally, we obtain that the defect is also a block invariant for cyclotomic Yokonuma--Hecke algebras.


\end{abstract}

\section{Introduction}

Iwahori--Hecke algebras associated with Weyl groups appear as endomorphism algebras of induced representations in the study of the representation theory of finite reductive groups. They can also be defined independently as one-parameter deformations of the Weyl group algebras.  Cyclotomic Hecke algebras generalise the notion of Iwahori--Hecke algebras from Weyl groups (and more generally, real reflection groups) to complex reflection groups, and they can be also seen as one-parameter deformations of the corresponding group algebras. Cyclotomic Hecke algebras associated with complex reflection groups appear naturally in the generalised Harish-Chandra series of finite reductive groups, as well as in the ordinary Harish-Chandra series of Spetses, the mysterious yet objects that are expected to be generalisations of finite reductive groups. Since their first appearance in the works of Brou\'e, Malle, Michel and Rouquier \cite{BM, BMR, BMM}, the study of cyclotomic Hecke algebras associated with complex reflection groups has grown as a subject on its own right, with numerous connections to other research areas in algebra, geometry, topology and even theoretical physics.

Let $W$ be a complex reflection group.
The irreducible representations of a cyclotomic Hecke algebra associated with $W$ are in bijection with the irreducible representations of $W$. When the parameter of the algebra is specialised to a complex number,  we obtain a decomposition matrix that records how the irreducible representations of the cyclotomic Hecke algebra (labelling the rows of the matrix) decompose into irreducible representations of the specialised algebra (labelling the columns of matrix). The decomposition matrix has a block diagonal form and we say that two irreducible representations of $W$ are in the same block if they label rows with non-zero entries in the same block of the matrix. 
If the specialised algebra is semisimple, then the decomposition matrix is the identity matrix.

The decomposition matrix is a tool used first by Brauer for the study of the representation theory of finite groups in characteristic $p>0$. If $G$ is a finite group and $V$ is an irreducible representation of $G$ over an algebraically closed field of characteristic $0$, with character $\chi_V$, then one can define the $p$-defect of $V$ to be the $p$-part of the integer $|G|/\chi_V(1)$. The $p$-defect of a block is the maximal $p$-defect of the representations it contains, and it is a good way to measure its complexity. Blocks of $p$-defect $0$ are simply singletons, while the structure of  blocks of $p$-defect $1$ is well-understood. However, the $p$-defect is not a block invariant, that is, not all representations in the same block have the same $p$-defect.

If $G$ is a finite group as above and $\Irr(G)$ denotes the set of its irreducible representations over an algebraically closed field of characteristic $0$, then the linear map $\tau:=\sum_{V\in \Irr(G)} ({\chi_V(1)}/{|G|})\chi_V$ is a symmetrising trace on $\Z[G]$, \emph{i.e.,} the bilinear map defined by $\tau$ is symmetric and non-degenerate. 
We say that $\Z[G]$ is a symmetric algebra and we call $\tau$ the canonical symmetrising trace on $\Z[G]$.
In the context of symmetric algebras, the integers $|G|/\chi_V(1)$ are the Schur elements with respect to $\tau$.

Let $W$ be a complex reflection group as before. A  generalisation of the canonical symmetrising trace of the group algebra of $W$ to the cyclotomic Hecke algebras associated with $W$ was conjectured to exist in \cite{BMM}. Its existence was a known fact for real reflection groups, but it remains an open problem for non-real ones. Any complex reflection group is isomorphic to a direct product of irreducible ones, which either belong to the infinite series $G(l,p,n)$ or to the exceptional groups $G_4, G_5,\ldots,G_{37}$,
following the classification by Shephard and Todd \cite{ShTo}. 
A list of groups for which the so-called ``BMM symmetrising trace conjecture'' has been proved is given in Subsection \ref{sec-Hecke}, right after stating the conjecture. However, the Schur elements with respect to the conjectural canonical symmetrising trace have been completely determined for all complex reflection groups. These Schur elements are products of $K$-cyclotomic polynomials, where $K$ is the splitting field of $W$. Using the bijection between the irreducible representations of $W$ and those of any cyclotomic Hecke algebra associated with $W$, we denote by $s_V$ the Schur element of $V \in \Irr(W)$. If $\Phi$ is a $K$-cyclotomic polynomial, then we define the $\Phi$-defect of $V$ to be the multiplicity of $\Phi$ as a factor of $s_V$.

Now, Schur elements yield a semisimplicity criterion for symmetric algebras: a symmetric algebra is semisimple if and only if all its Schur elements are non-zero. 
Therefore, if we specialise the parameter of a cyclotomic Hecke algebra to a complex number, and this complex number is not a root of unity (and more specifically, not a root of one of the $K$-cyclotomic polynomials appearing in the factorisation of its Schur elements), then the decomposition matrix is the identity matrix. If on the other hand we specialise the parameter to a root of a $K$-cyclotomic polynomial $\Phi$, we can define the $\Phi$-defect of a block to be the maximal $\Phi$-defect of the representations contained in the block.

In \cite{G} Geck showed that, contrary to $p$-defect for finite groups, the $\Phi$-defect is a block invariant for Iwahori--Hecke algebras of  Weyl groups in the equal parameter case, that is, all representations of a block share the same $\Phi$-defect. This result was generalised to all real reflection groups 
with the explicit calculation of the decomposition matrices for the exceptional groups in \cite{GP}. He  later conjectured that this result should also hold in the unequal parameter case for all real reflection groups. In the current paper, we 
prove that the $\Phi$-defect is a block invariant for all cyclotomic Hecke algebras associated with the complex reflection groups of the  infinite series $G(l,p,n)$, including thus the Weyl groups of type $B_n$ in the unequal parameter case. We also  computationally verify that the same result holds for all exceptional groups for which the decomposition matrices have been explicitly calculated. Based on all this evidence, we conjecture that the  defect is a block invariant for all cyclotomic Hecke algebras associated with complex reflection groups.

Hecke algebras associated with the groups of type $G(l,1,n)$ are also known as Ariki--Koike algebras. Their simple modules are parametrised by the $l$-partitions of $n$. Using the Schur element formula for the Ariki--Koike algebras that we developed in \cite{CJ}, we are able to give a combinatorial expression for the defect of the simple module $V^{\ulambda}$, where $\ulambda$ is an $l$-partition of $n$. 
This expression introduces the notion of charged hook lengths of a multipartition, which generalises the notion of generalised hook lengths introduced in \cite{CJ}.
Thanks to the Morita equivalence by Dipper and Mathas \cite{DiMa}, we can restrict to a specific setting, where the $\Phi$-defect of a module is given by the number of charged hook lengths that are congruent to $0$ modulo $e$, where $e$ is the order of the root of unity whose minimal polynomial is $\Phi$. We compute this number using the abacus decomposition of  $\ulambda$ and we prove that it is equal to the weight of $\ulambda$, another combinatorial datum attached to a multipartition by Fayers \cite{F}. 
Since the weight is a block invariant, we are able to deduce that the defect is a block invariant. 
This concludes the proof of our main result for the groups $G(l,1,n)$ (Theorem \ref{main-AK}), to which the whole Section 3 is devoted. At the same time, we obtain a new way of calculating the weight, using the easy-to-compute charged hook lengths. In Section 4, we extend the result of the defect being a block invariant to the groups $G(l,p,n)$ with the use of Clifford theory.

Furthermore, in Section 5, we prove that the defect is a block invariant for all exceptional cyclotomic Hecke algebras studied in \cite{CM}. These are the only cyclotomic Hecke algebras associated with exceptional complex reflection groups for which the decomposition matrices have been explicitly calculated (besides \cite{GP} where only real groups were considered). Our proof is computational and is based on a GAP3 program that we created to calculate all possible blocks and the defects of the representations they contain.

Finally, in Section 6, we generalise our results to the case of cyclotomic Yokonuma--Hecke algebras, introduced in \cite{CPA2} as generalisations of the Yokonuma--Hecke algebra of type $A$ and of the Ariki--Koike algebras. Through their connection to Ariki--Koike algebras, we are able to prove that the defect is a block invariant for these algebras as well. We conclude by  pondering the question whether the defect being a block invariant is a  property of essential algebras, that is, algebras whose Schur elements have a specific form (see Definition \ref{def-essential}) and include Hecke algebras and Yokonuma--Hecke algebras as particular cases. This is why, in Section 2, which contains all preliminary notions needed for  understanding this paper, we discuss Hecke algebras in the general context of essential algebras (which were in fact introduced in \cite{C} in order again to study the blocks of Hecke algebras in a more general context).

\section{Schur elements for Hecke algebras} 

In this section, we will give a quick overview of the theory of symmetric algebras and the related results on Hecke algebras. 
 
\subsection{Symmetric algebras}\label{sec2.1}
 Let $R$ be a commutative integral domain and let $A$ be an $R$-algebra, free and finitely generated as an $R$-module.  If $R'$ is a commutative integral domain containing $R$, we will write $R'A$ for $R' \otimes_R A$ 
and we will denote by $\mathrm{Irr}(R'A)$ the set of irreducible representations of $R'A$. 

Let $\mathcal{B} = (b_i)_{i \in I}$ be an $R$-basis of $A$.
A {\em symmetrising trace} on the algebra $A$ is a linear map $\tau : A \rightarrow R$ such that
the bilinear form $A \times A \rightarrow R,\, (a,b)\mapsto \tau(ab)$ is symmetric and non-degenerate, that is, the matrix $(\tau(b_ib_j))_{b_i,b_j \in \mathcal{B}}$ is symmetric and has a determinant in $R^\times$. If there exists a symmetrising trace on $A$, we say that $A$ is a {\em symmetric} algebra.

\begin{exa}{\rm Let $G$ be a finite group. The linear map $\tau: \Z[G] \rightarrow \Z$ defined by $\tau(1)=1$ and $\tau(g)=0$ for all $g \in G \setminus \{1\}$ is a symmetrising trace on $\Z[G]$; it is called the {\em canonical symmetrising trace} on $\Z[G]$.}
\end{exa}

Suppose that there exists a symmetrising trace $\tau$ on $A$.
Let $\mathcal{B}^\vee = (b_i^\vee)_{i \in I}$ denote the dual basis to $\mathcal{B}$ with  respect to $\tau$, which is uniquely determined by the property $\tau(b_i b_j^\vee) = \delta_{ij}$. 
Let $K$ be a field containing $R$ such that the algebra $KA$ is split. 
The map $\tau$ can be extended to $KA$ by extension of scalars. 
Let $V \in \Irr(KA)$ with  character $\chi_V$. The element $\sum_{i \in I} \chi_V(b_i)\, b_i^\vee$ belongs to the centre of $KA$ \cite[Lemma 7.1.7]{GP} and, by Schur's lemma, acts as a scalar on $V$. This scalar, denoted by $s_V$, is called 
the {\em Schur element} associated with $V$.  We have $s_V \in R_K$, where $R_K$ denotes the integral closure of $R$ in $K$  \cite[Proposition 7.3.9]{GP}. 
Moreover,  the Schur element $s_V$ satisfies the following equation (which also yields a formula for calculating $s_V$ when ${\rm char}K\not|\,\, \chi_V(1)$):
$$
s_V \,\chi_V(1) =\sum_{i \in I} \chi_V(b_i)\, \chi_V(b_i^\vee) \ .
$$
The algebra $KA$ is semisimple if and only if $s_V \neq 0$ for all $V \in \Irr(KA)$ \cite[Theorem 7.2.6]{GP}.
If this is the case, we  have (see \cite{CuRe} or \cite[Theorem 7.2.6]{GP}):
$$
\tau = \sum_{V \in \Irr(KA)}\frac{1}{s_V}\chi_V.
$$

\begin{exa}\label{group example}
{\rm Let $G$ be a finite group and let $\tau$ be the canonical symmetrising trace on $A:=\Z[G]$. 
The set $\{g\}_{g \in G}$  forms a basis of $A$ over $\Z$, with $\{g^{-1}\}_{g \in G}$ the dual basis of $A$ with respect to $\tau$.
If $K$ is an algebraically closed field of characteristic $0$, then $KA$ is a split semisimple algebra and $s_V = |G|/\chi_V(1) \in \Q$ for all $V \in \Irr(KA)$. Because of the integrality of the Schur elements, we must have $|G|/\chi_V(1) \in  \Z_K \cap \Q = \Z$ for all $V \in \Irr(KA)$. Thus, we have also recovered the fact that $\chi_V(1)$ divides $|G|$ in $\Z$.}
\end{exa}

From now on, we assume that $R$ is integrally closed in $K$ and that $KA$ is split semisimple.
Let $\theta : R \rightarrow L$ be a ring homomorphism into a field $L$ such that $L$ 
is the field of fractions of $\theta(R)$. We call such a ring homomorphism a {\em specialisation} of $R$. 
Let us also assume that the algebra $LA:=L \otimes_R A$ is split. 
We then obtain a decomposition matrix $D_\theta=([V:M])_{V \in \Irr(KA),\,M \in \Irr(LA)}$ that records how the simple modules of $KA$ decompose after the specialisation $\theta$ (for the definition of the decomposition matrix, the reader may refer to \cite[Theorem 7.4.3]{GP}). 
The \emph{Brauer graph} associated with $\theta$ has vertices labelled by $\Irr(KA)$ and an edge joining $V,\,V' \in \Irr(KA)$ if there exists  $M\in \Irr(LA)$ such that $[V:M] \neq 0$ and  $[V':M]\neq  0$. A connected component of the Brauer graph is called a $\theta$-\emph{block}.

Since $L$ is a field, the kernel of $\theta$ is a prime ideal of $R$. Let $\mathcal{O}=\{a/b\,|\, a,b \in R,\,\theta(b) \neq 0\}$ be the localisation of $R$ at ${\rm Ker} \theta$. We say that $\theta$ is a \emph{principal specialisation} if $\mathcal{O}$ is a discrete valuation ring. Note that $\theta$ extends naturally to a map from $\mathcal{O}$ to the residue field of $\mathcal{O}$ and that the decomposition matrix does not change if we consider our algebra $A$ over $\mathcal{O}$ instead of $R$. If $\Phi$ is a generator of the maximal ideal of $\mathcal{O}$, we will write ${\nu}_\Phi: K^{\times} \rightarrow \mathbb{Z}$ for the exponantial valuation associated with $\mathcal{O}$. Thus, any $x \in K^{\times}$ can be written uniquely in the form $x=\Phi^{{\nu}_{\Phi}(x)}u$, where $u \in \mathcal{O}^\times$. Furthermore, we have $\mathcal{O}=\{x \in K^\times\,|\, {\nu}_\Phi(x) \geq 0\} \cup \{0\}$.

Let $V \in \Irr(KA)$. Recall that $s_V \neq 0$. The integer ${\nu}_\Phi(s_V)$ is called the \emph{$\Phi$-defect} of $V$. 
We define the $\Phi$-defect of a  $\theta$-block $B$ to be 
the highest among the values of $\Phi$-defects of the elements of $B$. The following result about blocks of $\Phi$-defect $0$ is a consequence of \cite[Theorem 7.5.11]{GP} and \cite[Proposition 4.4]{GR}:

\begin{Th} {\bf (Blocks of defect 0)} Let $V \in \Irr(KA)$. We have that $\theta(s_V) \neq 0$ if and only if $\{V\}$ is a $\theta$-block and the corresponding decomposition matrix block is the $(1 \times 1)$-matrix with entry $1$.
\end{Th}

In particular, we recover the following known semisimplicity criterion for the algebra $LA$ \cite[Theorem 7.4.7]{GP}: $LA$ is semisimple if and only if $\theta(s_V) \neq 0$ for all $V \in \Irr(KA)$. If this is the case, Tits's deformation theorem (see, for example, \cite[Theorem 7.4.6]{GP}) yields a bijection between $ \Irr(KA)$ and $ \Irr(LA)$.

 \subsection{Essential algebras}\label{sec-essential} The notion of essential algebras was introduced by the first author in \cite{C} in order to study the block theory of Hecke algebras in positive characteristic. Let 
 $A$ be an $R[\textbf{x},\textbf{x}^{-1}]$-algebra, where $R$ is a Noetherian integrally closed domain and $\textbf{x}=(x_j)_{0 \leq j \leq m-1}$ is a set of $m$ indeterminates over $R$. We assume that $A$ is free and finitely generated as an $R[\textbf{x},\textbf{x}^{-1}]$-module and that $A$ is symmetric. We also assume that the algebra
 $K(\textbf{x})A$ is split semisimple, where $K$ is the field of fractions of $R$. 

\begin{Def}\label{def-essential}
We say that the algebra $A$ is \emph{essential} if, for each $V \in \textrm{Irr}(K(\textbf{x})A)$, 
the Schur element $s_V(\textbf{x})$ is of the form
\begin{equation}\label{schurfactors}
s_V(\textbf{x})= \xi_V N_V \prod_{i \in I_V} \Psi_{V,i}(M_{V,i})^{n_{V,i}}
\end{equation}
where
\begin{enumerate}[(a)]
    \item $\xi_V$ is an element of $R \setminus\{0\}$,\smallbreak
    \item $N_V$ is a monomial in $R[\textbf{x},\textbf{x}^{-1}]$,\smallbreak
    \item $I_V$ is a finite index set,\smallbreak
    \item $(\Psi_{V,i})_{i \in I_V}$ is a family of  monic polynomials  in one variable with coefficients in $R$, irreducible over $K$, such that $\Psi_{V,i}(0) \in R^\times$,\smallbreak
        \item $(M_{V,i})_{i \in I_V}$ is a family of \emph{primitive} monomials in $R[\textbf{x},\textbf{x}^{-1}]$, that is, if $M_{V,i} = \prod_{j=0}^{m-1} x_j^{a_j}$, then $\textrm{{gcd}}(a_j)=1$,\smallbreak
     \item  ($n_{V,i})_{i \in I_V}$ is a family of positive integers.
    \end{enumerate}
    \end{Def}
    
By \cite[Theorem 1.5.6]{C}, the Laurent polynomials $ \Psi_{V,i}(M_{V,i})$ are irreducible in $K[\textbf{x},\textbf{x}^{-1}]$, so Equation \eqref{schurfactors} yields the (unique) factorisation of $s_V(\textbf{x})$ inside that ring. What is more, the monomials $(M_{V,i})_{i \in I_V}$ are unique up to inversion \cite[Proposition 3.1.2]{C}.

Let $y$ be another indeterminate over $R$, and let $\varphi: R[\textbf{x},\textbf{x}^{-1}] \rightarrow R[y,y^{-1}]$  be an $R$-algebra morphism  such that $\varphi(x_j)=y^{r_j}$, where $r_j \in \Z$ for all $j=0,1,\ldots,m-1$.  We denote by $A_\varphi$ the algebra obtained as a specialisation of $A$ via the morphism $\varphi$.
As long as $\varphi(s_V(\textbf{x}))\neq 0$ for all $V \in \textrm{Irr}(K(\textbf{x})A)$, the algebra 
$K(y)A_\varphi$ is split semisimple (this is quite straightforward, but one can also refer to the proof of \cite[Proposition 4.3.3]{C}). If this is the case, then, by Tits's deformation theorem, we have a bijection between
$\Irr(K(\textbf{x})A)$ and $\Irr(K(y)A_\varphi)$.
Moreover, the specialisation of the symmetrising trace of $A$ via $\varphi$ is a symmetrising trace on $A_\varphi$. In particular, the algebra $A_\varphi$ is essential.
If we now consider a specialisation $\theta$ of $A_\varphi$ as in \S \ref{sec2.1}, then the representation theory of the specialised algebra becomes ``interesting'' when $\theta(y)$ is a root of one of the irreducible factors of the Schur elements of $A_\varphi$ (but not only then).

\begin{Rem}
In fact, the algebra $K(y)A_\varphi$ is not semisimple if and only if 
$\varphi(\Psi_{V,i}(M_{V,i}))=0$ for some $V \in \textrm{Irr}(K(\textbf{x})A)$ and some $i \in I_V$. This in turn can happen if and only if $\varphi(M_{V,i})=1$ and $\Psi_{V,i}(1) =0$. This is why in our original definition of essential algebras \cite[Definition 3.1.1]{C}, we also asked that the polynomials $\Psi_{V,i}$ satisfy $\Psi_{V,i}(1) \neq 0$.
\end{Rem}

\subsection{Hecke algebras}\label{sec-Hecke}

Let $\mathcal{V}$ be a finite dimensional complex vector space. A \emph{pseudo-reflection} is a non-trivial element $s \in  \mathrm{GL}(\mathcal{V})$ that fixes a hyperplane pointwise,
that is, ${\rm dim}({\rm Ker}(s - {\rm id}_\mathcal{V}))={\rm dim}(\mathcal{V})-1$. The hyperplane ${\rm Ker}(s - {\rm id}_\mathcal{V})$ is the \emph{reflecting hyperplane} of $s$.
A {\em complex reflection group} is a finite subgroup of $\mathrm{GL}(\mathcal{V})$ generated by pseudo-reflections. If $W$ acts irreducibly on $\mathcal{V}$, then $W$ is called an irreducible complex reflection groupe and the dimension of $\mathcal{V}$ is called the {\em rank} of $W$. Every complex reflection group is isomorphic to a direct product of irreducible complex reflection groups; the latter have been classified by Shephard and Todd \cite{ShTo}:

\begin{Th}\label{ShToClas} Let $W \subset \mathrm{GL}(\mathcal{V})$ be an irreducible complex
reflection group. Then 
\begin{itemize}
  \item either $(W,\mathcal{V}) \cong (G(l,p,n),\C^{n-\delta_{l,1}})$, where $(l,p,n,l/p) \in (\N^*)^4 \setminus \{(2,2,2,1)\}$ and $G(l,p,n)$ is the group of all 
  $n \times n$ monomial matrices whose non-zero entries are ${l}$-th roots of unity, while the product of all non-zero
  entries is an $(l/p)$-th root of unity. \smallbreak
  \item $(W,\mathcal{V})$ is isomorphic to one of the 34 exceptional groups
  $G_n$ $(n=4,\ldots,37)$.
\end{itemize}
\end{Th}

From now on, let $W$ be an irreducible complex reflection group. Let $\mathcal{A}$ be the set of reflecting hyperplanes of $W$ and let $\mathcal{V}^{\textrm{reg}}:= \mathcal{V}\setminus \bigcup_{H\in \mathcal{A}} H$. 
We define $P(W) := \pi_1(\mathcal{V}^{\textrm{reg}}, x_0)$ and $B(W) := \pi_1(\mathcal{V}^{\textrm{reg}}/W, x_0)$, where $x_0 \in \mathcal{V}^{\textrm{reg}}$ is some fixed basepoint, to be respectively the {\em pure braid group} and
the {\em braid group} of $W$. 
It is known by \cite[Theorem 12.8]{Bes2} that
the centre of $B(W)$ is cyclic, generated by some element $\boldsymbol{\rm z}$. We set $\boldsymbol{\pi}: = \boldsymbol{\rm z}^{|Z(W)|} \in P(W)$, where $Z(W)$ denotes the centre of $W$.
For every orbit $\mathcal{C}$ of the action of $W$ on $\mathcal{A}$, let
$e_{\mathcal{C}}$ be the common order of the subgroups $W_H$, where $H$
is any element of $\mathcal{C}$ and $W_H$ is the pointwise stabiliser of $H$. Note that $W_H$ is cyclic, for all $H \in \mathcal{A}$. Set 
$R:=\mathbb{Z}[\textbf{u},\textbf{u}^{-1}]$ to be the Laurent polynomial ring
in a set of indeterminates $\textbf{u}=(u_{\mathcal{C},j})_{(\mathcal{C} \in
\mathcal{A}/W)(0\leq j \leq e_{\mathcal{C}}-1)}$. The \emph{generic
Hecke algebra} \index{generic Hecke algebra} $\mathcal{H}(W)$ of $W$ is  the quotient of the group
algebra $R[B(W)]$ by the ideal
generated by the elements of the form
$$(s-u_{\mathcal{C},0})(s-u_{\mathcal{C},1}) \cdots (s-u_{\mathcal{C},e_{\mathcal{C}}-1}),$$
where $\mathcal{C}$ runs over the set $\mathcal{A}/W$ and
$s$ runs over the set of monodromy generators around 
the images in $\mathcal{V}^{\textrm{reg}}/W$ of 
the elements of $\mathcal{C}$ (see \cite[\S 2]{BMR} for their definition).

This definition of generic Hecke algebras of complex reflection groups is due to Brou\'e, Malle and Rouquier,  who also conjectured that the algebra 
$\mathcal{H}(W)$ is a free $R$-module of rank  $|W|$ \cite[\S 4]{BMR}. This conjecture is known as the ``BMR freeness conjecture'', and was then known to hold for the real reflection groups \cite[IV, \S 2]{Bou05}, the groups $G(l,p,n)$ \cite{ArKo, BM, Ar} and $G_4$ \cite{BM}. The proof of this conjecture for the exceptional irreducible complex reflection groups was completed only a couple of years ago, thanks to the works of Chavli \cite{Ch18, Ch17}, Marin \cite{Mar41, Mar43, MarNew}, Marin--Pfeiffer \cite{MaPf} and Tsuchioka \cite{Tsu} .

Let now $K$ be the field generated by the traces on $\mathcal{V}$ of all the elements of $W$. 
Benard \cite{Ben} and Bessis \cite{Bes1} have proved, using a case-by-case
analysis, that $K$ is a splitting field for $W$. The field $K$  is called the \emph{field of
definition} of $W$.  If $K \subseteq \mathbb{R}$, then $W$ is a finite Coxeter group, and
 if $K=\mathbb{Q}$, then $W$ is a Weyl group.
Malle \cite[5.2]{Ma4} has shown that, given that the BMR freeness conjecture holds, we can always find $N_W \in \Z_{>0}$ such that  if we take $\textbf{{v}}:=(v_{\mathcal{C},j})_{(\mathcal{C} \in
\mathcal{A}/W)(0\leq j \leq e_{\mathcal{C}}-1)}$ defined by
\begin{equation}\label{NW}
v_{\mathcal{C},j}^{N_W}:= \eta_{e_\mathcal{C}}^{-j}u_{\mathcal{C},j} ,
\end{equation}
where 
$\eta_{e_\mathcal{C}}:=\exp(2\pi i/e_\mathcal{C})$, then the 
$K(\textbf{{v}})$-algebra
$K(\textbf{{v}})\mathcal{H}(W)$ is split semisimple.
 Taking $N_W$ to be the number of roots of unity in $K$ works every time, but sometimes it is enough to take $N_W$ to be even as small as $1$ (for example, if $W=G(l,1,n)$ or $W=G_4$).
Following Tits's deformation theorem,  the specialisation $v_{\mathcal{C},j}\mapsto 1$ induces a
bijection between
$\mathrm{Irr}(K(\textbf{v})\mathcal{H}(W))$ and the set $\mathrm{Irr}(W)$ of irreducible representations of $W$. 

The following conjecture is \cite[2.1]{BMM}:

\begin{conj}\label{BMM sym}\
\textbf{``The BMM symmetrising trace conjecture''} There exists a symmetrising trace $\tau: \mathcal{H}(W) \rightarrow R$ that satisfies the following two conditions:
\begin{enumerate}
\item $\tau$ specialises to the canonical symmetrising trace on the group algebra of $W$ when  $u_{\mathcal{C},j}\mapsto \eta_{e_\mathcal{C}}^{j}$. \smallbreak 
\item  $\tau$ satisfies
$$
\tau(T_{\beta^{-1}})^* =\frac{\tau(T_{\beta\boldsymbol{\pi}})}{\tau(T_{\boldsymbol{\pi}})} \quad \text{ for all } \beta \in B(W),
$$
where $\beta \mapsto T_{\beta}$ denotes the natural surjection $R[B(W)] \rightarrow \mathcal{H}(W)$ and
$x \mapsto x^*$ the automorphism of  $R$ given by $\textbf{\em u} \mapsto \textbf{\em u}^{-1}$.
\smallbreak
\end{enumerate}
\end{conj}

If $\tau$ exists, then $\tau$ is unique \cite[2.1]{BMM}, and it is called the \emph{canonical symmetrising trace on} $\mathcal{H}(W)$. 
When Brou\'e, Malle and Michel stated the above conjecture, it was known to hold
for real reflection groups, while a symmetrising trace satisfying Condition (1) existed for the complex reflection groups of the infinite series $G(l,p,n)$ (defined by Bremke and Malle in \cite{BreMa} and shown to be non-degenerate over $R$ by Malle and Mathas in \cite{MaMa}).
 The non-real exceptional complex reflection groups  for which the BMM symmetrising trace conjecture is known to hold are: 
\begin{itemize}
\item $G_4$ \cite{MM10, Mar46, BCCK} (3 independent proofs),  \smallbreak
\item $G_5$, $G_6$, $G_7$, $G_8$ \cite{BCCK},\smallbreak
\item  $G_{12}$, $G_{22}$, $G_{24}$ \cite{MM10}, \smallbreak
\item $G_{13}$ \cite{BCC}.
\end{itemize}

Nevertheless, the Schur elements with respect to the (conjectural) canonical symmetrising trace have been determined for all complex reflection groups. For real reflection groups, a complete list of bibliographical references can be found in \cite[\S 2.2]{HDR}. For the groups of type $G(l,1,n)$, two independent descriptions of the Schur elements with respect to the symmetrising  trace of Bremke and Malle have been given by Geck--Iancu--Malle \cite{GIM} and Mathas \cite{Mat}, whereas a simpler description has been subsequently given in \cite{CJ} using Mathas's formula. 
Moreover, Geck, Iancu and Malle have shown that these Schur elements satisfy a certain palindromicity property \cite[Theorem 5.2]{GIM}, which, by \cite[Lemma 2.7]{BMM}, amounts to proving Condition (2) of the BMM symmetrising trace conjecture.
From the Schur elements of $G(l,1,n)$, one recovers the Schur elements of $G(l,p,n)$ when $n > 2$ or $n = 2$ and $p$ is odd, with the use of Clifford theory. 
All remaining irreducible complex reflection  groups have been dealt with by Malle: the groups of rank $2$ in \cite{Ma2}, the groups of superior rank in \cite{Ma5}. 

Using a case-by-case analysis, we have shown that the algebra $\mathcal{H}(W)$, defined over the ring $\mathbb{Z}_{K}[\textbf{v},\textbf{v}^{-1}]$, is essential \cite[Theorem 4.2.6]{C}. In fact, the irreducible polynomials $\Psi_{V,i}$ appearing in the factorisation of the Schur elements of $\mathcal{H}(W)$ are always $K$-cyclotomic polynomials, that is, minimal polynomials of roots of unity over the field $K$ \cite[Theorem 4.2.5]{C}. They also satisfy $\Psi_{V,i}(1) \neq 0$.

%


Let $y$ be an indeterminate, and let $\varphi: \mathbb{Z}_{K}[\textbf{v},\textbf{v}^{-1}] \rightarrow \mathbb{Z}_{K}[y,y^{-1}]$   be a $\Z_K$-algebra morphism such that $\varphi(v_{\mathcal{C},j})=y^{r_{\mathcal{C},j}}$, where $r_{\mathcal{C},j} \in \Z$ for all $\mathcal{C},j$. We call $\varphi$ a \emph{cyclotomic specialisation} of $\mathcal{H}(W)$ and the algebra $\mathcal{H}_\varphi(W)$ a  {\em cyclotomic Hecke algebra}. By \cite[Proposition 4.3.3]{C}, the algebra $K(y)\mathcal{H}_\varphi(W)$ is 
split semisimple (since none of the Schur elements is mapped to $0$ by $\varphi$), and by Tits's deformation theorem, 
we have: 
$$\begin{array}{ccc}
    \textrm{Irr}(K(\textbf{v})\mathcal{H}(W)) & \leftrightarrow & \textrm{Irr}(K(y)\mathcal{H}_\varphi(W)). 
  \end{array}$$
  Furthermore, the
specialisation $y \mapsto 1$ yields a bijection:
  $$\begin{array}{ccc}
   \textrm{Irr}(K(y)\mathcal{H}_\varphi(W)) & \leftrightarrow & \textrm{Irr}(W). 
  \end{array}$$
  
 \begin{Rem}
Let $\zeta \in K$ be a root of unity and set $x:=\zeta y$. The cyclotomic specialisation $\varphi$ can be also given by  $\varphi(v_{\mathcal{C},j})=(\zeta^{-1}x)^{r_{\mathcal{C},j}}$ for all $\mathcal{C},j$. Then the algebra $\mathcal{H}_\varphi(W)$ can be regarded as a $\mathbb{Z}_{K}[x,x^{-1}]$-algebra, which specialises to the group algebra of $W$ when $x \mapsto \zeta$. If that is the case, we refer to $\varphi$ as a 
$\zeta$-\emph{cyclotomic specialisation}. 
\end{Rem}

The specialisation of the canonical symmetrising trace on $\mathcal{H}(W)$ is a symmetrising trace on $\mathcal{H}_\varphi(W)$, whose Schur elements are the images of the Schur elements of $\mathcal{H}(W)$ via $\varphi$. Therefore, the form of the Schur elements of $\mathcal{H}_\varphi(W)$ is given by the following proposition \cite[Proposition 4.3.5]{C}:

\begin{Prop}\label{Schur element cyclotomic}
Let $V \in \Irr(W)$. The associated  Schur element $s_V(y)$ of $\mathcal{H}_\varphi(W)$
is of the form
\begin{equation}\label{phifactors}
s_V(y)=\xi_V y^{a_V} \prod_{\Phi \in
C_V}\Phi(y)^{n_{V,\Phi}}
\end{equation}
 where $\xi_V \in
\mathbb{Z}_K$, $a_V\in \mathbb{Z}$, $C_V$ is a finite set of $K$-cyclotomic polynomials, and $n_{V,\Phi} \in
\mathbb{N}$.
\end{Prop}

Let now $\eta$ be a non-zero complex number, and let $\theta: \mathbb{Z}_{K}[y,y^{-1}] \rightarrow K(\eta), y \mapsto \eta$ be a specialisation of $\mathbb{Z}_{K}[y,y^{-1}]$. The specialised algebra $K(\eta) \mathcal{H}_\varphi(W)$ is split. It is also semisimple unless $\eta$ is the root of one of the $K$-cyclotomic polynomials appearing in the factorisation of the Schur elements of $\mathcal{H}_\varphi(W)$. 
If $s_V(y)$ is a Schur element of $\mathcal{H}_\varphi(W)$,  with a  form given by Equation \eqref{phifactors}, and $\eta$ is a root of some $\Phi \in C_V$, then $n_{V,\Phi} = \nu_\Phi(s_V(y))$, the $\Phi$-defect of $V$. 
Furthermore, if $\Phi=\Phi_e$, the $e$-th cyclotomic polynomial over $\Q$ for some $e > 0$, then we may also refer to the $\Phi_e$-defect of $V$ as the $e$-\emph{defect} of $V$.

\begin{Rem}
We have that $s_V(1)$ is the Schur element of $V \in \Irr(W)$ with respect to the canonical symmetrising trace on $\Z[W]$ (see Example \ref{group example}). Thus, $s_V(1) \neq 0$ and $\Phi_1(y)$ is not a factor of $s_V(y)$. We deduce that  the $1$-defect of $V$ is $0$ for all $V \in \Irr(W)$.
\end{Rem}

Geck has shown that if $W$ is a finite Coxeter group and we are in the equal parameter case (that is, $r_{\mathcal{C},j}=r_{\mathcal{C'},j}$ for all $\mathcal{C},\mathcal{C}' \in \mathcal{A}/W$ and $j=0,1$), then any two representations that are in the same $\theta$-block have the same $\Phi$-defect. This is proved in \cite[Proposition 7.4]{G} for Weyl groups, using an analogous result for blocks of finite groups of Lie type \cite[1.4]{G90}, and it can be verified by inspection  for $H_3$ and $H_4$, using the explicit knowledge of the $\theta$-blocks given in \cite[Appendix F]{GP}. In \cite[Remark 3.3.16]{GJ}, it is conjectured that such a property may hold for all cyclotomic Hecke algebras associated to finite Coxeter groups. We will go one step further and conjecture that this property holds for all cyclotomic Hecke algebras associated to complex reflection groups.

\begin{conj}\label{our conj}
Let $V,V' \in  \Irr(W)$. If $V$ and $V'$ belong to the same $\theta$-block of $K(\eta) \mathcal{H}_\varphi(W)$, then they have the same $\Phi$-defect.
\end{conj}

In the rest of this paper, we are going to prove the above conjecture in the following cases:\smallbreak
\begin{itemize}
\item $W=G(l,p,n)$, where $n \neq 2$, or $n=2$ and $p$ is odd;\smallbreak
\item $W$ is an exceptional complex reflection group of rank $2$ and $\mathcal{H}_\varphi(W)$ has ``distinguished'' parameters motivated by generalised Harish-Chandra theory.
\end{itemize}
These cases cover the unequal parameter case for finite Coxeter groups of type $B_n \cong G(2,1,n)$, as well as the dihedral groups $I_2(p) \cong G(p,p,2)$ where $p$ is odd.
We will also show that the validity  of the conjecture transfers to some other types of essential algebras, the cyclotomic  Yokonuma--Hecke algebras, which are isomorphic to 
direct sums of matrix algebras over tensor products of cyclotomic Hecke algebras.

\begin{Rem}
Let $W$ be a Weyl group, let $V \in   \Irr(W)$ and let $p$ be a prime number. As we mentioned in the introduction, when studying the representation theory of $W$ over a field $\mathbb{F}$ of characteristic $p$, we are interested in the traditional $p$-defect of $V$, which is defined as the $p$-part of the integer $|W|/\chi_V(1)$ and is not a block invariant. Since $|W|/\chi_V(1)=s_V(1)$ and $K=\Q$, that number is equal to the sum of
$\nu_{\Phi_{p^k}}(s_V(y))$ for all $k \in \mathbb{Z}_{>0}$. If now $\nu_{\Phi_{p^k}}(s_V(y))=0$ for all $k \in \mathbb{Z}_{>1}$, then that $p$-defect coincides with the $\Phi_p$-defect of $V$, which is a block invariant of the Hecke algebra when $y$ specialises to a primitive $p$-th root of unity  (or conjectured to be for $I_2(6)$ in the unequal parameter case). For example, if $W$ is the symmetric group $\mathfrak{S}_n$, then this condition holds if and only if $p^2>n$ (for the form of the Schur elements in this case, see Formula \eqref{claim} for $l=1$).  Note  that James's conjecture \cite{Jam} also makes a connection between the decomposition matrix of the group algebra of $\mathfrak{S}_n$ over $\mathbb{F}$ and the one of its corresponding cyclotomic Hecke algebra when $y$ specialises to a primitive $p$-th root of unity with the assumption $p^2>n$ (the conjecture has been recently disproved, cf.~ \cite{Wi}). 
Furthermore, in Theorem \ref{weight-defect}, we establish that the $\Phi_p$-defect is equal to the $p$-weight, which is also a known block invariant of $\mathbb{F}[\mathfrak{S}_n]$.
\end{Rem}

\section{Defect in cyclotomic Ariki--Koike algebras}\label{sec-AK}

Hecke algebras associated with complex reflection groups of type $G(l,1,n)$ are also known as Ariki--Koike algebras.
In this section, we give the description of the Schur elements for Ariki-Koike algebras obtained in \cite{CJ}. This description uses a generalised version of hook lengths for partitions called ``generalised hook lengths''. 
We give an easy way to compute variations of these elements, called ``charged hook lengths'', using the notion of abaci. 
We produce counting formulas for the number of charged hook lengths equal to $0$ in $\Z$ and in $\Z/e\Z$ for  $e \in \Z_{>1}$,  which allow us to establish a connection between the notions of defect and weight for a multipartition (see \S \ref{sec-def-wei}). This connection leads to the proof of Conjecture \ref{our conj} for type $G(l,1,n)$.

Throughout this section, let  $n\in \mathbb{Z}_{\geq 0}$ and $l\in \mathbb{Z}_{>0}$. 

\subsection{Generic Ariki-Koike algebras}
Let ${\bf q}:=(Q_0,\,\ldots,\,Q_{l-1}\,;\,q)$ be a set of $l+1$ indeterminates  and
set $R:=\mathbb{Z}[{\bf q},{\bf q}^{-1}]$.
The {\it Ariki-Koike algebra} $\mathcal{H}^{\bf q}_{n}$  is the associative $R$-algebra (with unit) with generators $T_0,\,T_1,\,\ldots,\,T_{n-1}$ and relations:
\begin{center}
$\begin{array}{rl}
(T_0 -Q_0) (T_0 -Q_1)\cdots(T_0 -Q_{l-1})=0& \\  \smallbreak
(T_i-q)(T_i+1)=0  & \text{for $1\leq i \leq n-1$}\\  \smallbreak
T_0T_1T_0T_1=T_1T_0T_1T_0&\\  \smallbreak
T_iT_{i+1}T_i=T_{i+1}T_iT_{i+1} & \text{for $1\leq i \leq n-2$}\\  \smallbreak
T_iT_j=T_jT_i  &\text{for  $0\leq i <j \leq n-1$ with $j-i>1$.}
\end{array}$
\end{center}
The generic Hecke algebra of $G(l,1,n)$ has a presentation very similar to the presentation of the Ariki--Koike algebra, with the only difference being that the generators $T_i$, for $i \neq 0$, satisfy quadratic relations of the form $(T_i-q_0)(T_i-q_1)=0$, where $q_0, q_1$ are two indeterminates. If we set $q:=-q_0q_1^{-1}$, then 
$\mathcal{H}(G(l,1,n))$ can be defined over $R$ and is isomorphic to $\mathcal{H}^{\bf q}_{n}$.


The representation theory of $\mathcal{H}^{\bf q}_{n}$ has first been studied by Ariki and Koike \cite{ArKo} and is governed by the combinatorics of partitions.   
   A {\it partition} $\lambda$ is a   non-increasing
sequence $\lambda=(\lambda_{1},\cdots,\lambda_{m})$ of non-negative
integers. One can assume this sequence is infinite by adding parts equal to
zero. The {\it rank}  of the partition is defined to be the number 
$|\lambda|:=\sum_{1\leq i\leq m} \lambda_i$. If $|\lambda|=n \in \mathbb{N}$,
 we say that $\lambda$ is a \emph{partition of} $n$. By convention, the unique partition of $0$ is the empty partition $\emptyset$.
More generally, for $l\in \mathbb{Z}_{>0}$, an {\it $l$-partition $\ulambda$ of} $n$ is a sequence $(\lambda^0,\lambda^1,\ldots,\lambda^{l-1})$  of $l$ partitions 
 such that $\sum_{0\leq j\leq l-1} |\lambda^j |=n$. The number $n$ is  called  the \emph{rank} of $\ulambda$ and it is denoted by $|\ulambda|$.  The set of $l$-partitions is denoted by $\Pi^l$ and the 
  set of $l$-partitions of rank $n$ is denoted by $\Pi^l (n)$ (if $l=1$, the letter $l$ is omitted). 
It follows from \cite{ArKo} and Ariki's semisimplicity criterion \cite{Arsem} that the algebra $\Q({\bf q})\mathcal{H}^{\bf q}_{n}$ is split semisimple. We have a bijection $\Pi^l (n) \leftrightarrow \Irr(\Q({\bf q})\mathcal{H}^{\bf q}_{n}),\, \ulambda  \mapsto V^{\ulambda}$.

\subsection{Generalised hook lengths and Schur elements}
As stated in the previous section, two independent descriptions of the Schur elements of $\mathcal{H}^{\bf q}_{n}$  have been given by Geck--Iancu--Malle  \cite{GIM} and Mathas \cite{Mat}. In both articles, the Schur elements are given as fractions in $\Q({\bf q})$. However, since the Schur elements belong to the Laurent polynomial ring $R$, we know that the denominator  always divides the numerator. In \cite{CJ} we have given a cancellation-free formula for these Schur elements, that is, we have explicitly described their irreducible factors in $R$. This is the formula that we are going to use in this paper. In order to present it, we need some further combinatorial notions.

Let $\lambda \in \Pi$. We define the set of nodes $[\lambda]$ of $\lambda$ to be the set
$$[\lambda]:=\{(i,j)\,\,|\,\, i\geq 1,\,\,1 \leq j \leq \lambda_i\}.$$
Each node $(i,j)$ represents a box in the $i$-th row and the $j$-the column of the \emph{Young diagram} of $\lambda$, which we define as a left-justified array. 
We identify partitions with their Young diagrams.

The {conjugate partition} of $\lambda$ is the partition $\lambda'$ defined by
$$\lambda'_{k}:=\sharp\{i\,|\,i\geq 1 \text{ such that } \lambda_i\geq k\}.$$
The set  of nodes of $\lambda'$ satisfies
$$(i,j) \in [\lambda'] \Leftrightarrow (j,i)\in [\lambda]$$
and the Young diagram of $\lambda'$ is obtained from the one of $\lambda$ by transposition with respect to the main diagonal.

If $x=(i,j) \in [\lambda]$ and $\mu \in \Pi$, we define the {\em generalised hook length of $x$ with respect to } $(\lambda, \mu)$ to be the integer:
\begin{equation}\label{ghl}
h_{i,j}^{\lambda,\mu}:=\lambda_i-i+\mu'_j-j+1.
\end{equation}
For $\mu=\lambda$, the above formula becomes the classical hook length formula, which yields the length of the hook of $\lambda$ that $x$ belongs to (the hook consists of $x$, the boxes below $x$, and the boxes to its right).

Finally, we set
$$N(\lambda):=\sum_{i \geq 1}  (i-1)\lambda_i
 =  \frac{1}{2}\sum_{i \geq 1}(\lambda'_i-1)\lambda'_i
=\sum_{i\geq 1} \left( \begin{array}{c} \lambda'_i \\ 2 \end{array}\right). 
$$ 
The result below is the main result of \cite{CJ}.

\begin{Th}\label{canfreeform} Let $\ulambda=(\lambda^{0},\lambda^{1},\ldots,\lambda^{l-1}) \in \Pi^l (n)$.  The Schur element $s_{\ulambda}({\bf q})$ associated to the simple module $V^{\ulambda}$ is given by
$$s_{\ulambda}({\bf q})=(-1)^{n(l-1)}q^{-N(\bar{\ulambda})}(q-1)^{-n} \prod_{0 \leq a \leq l-1} \prod_{(i,j) \in [\lambda^{a}]}\prod_{0\leq b \leq l-1}  (q^{h_{i,j}^{\lambda^a,\lambda^{b}}}Q_aQ_b^{-1}-1)$$
where $\bar{\ulambda}$ is the partition of $n$ obtained by reordering all the numbers in $\ulambda$. If we want to get rid of the term $(q-1)^{-n}$, we can rewrite the formula as follows:
\begin{equation}\label{claim}
s_\ulambda({\bf q})=(-1)^{n(l-1)}q^{-N(\bar{\ulambda})}
 \prod_{0 \leq a \leq l-1} \prod_{(i,j) \in [\lambda^{a}]}\left( [{{h_{i,j}^{\lambda^a,\lambda^{a}}}}]_q
 \prod_{0 \leq b \leq l-1,\, b\neq a} (q^{h_{i,j}^{\lambda^a,\lambda^{b}}}Q_aQ_b^{-1}-1)\right)
 \end{equation}
 where $[k]_q
:=(q^k-1)/(q-1)=q^{k-1}+q^{k-2}+\cdots+q+1$ for any $k \in \N^*$.
\end{Th}

\begin{Rem}\label{tobeusedlater}
Let $\lambda \in \Pi$. The \emph{content} of a node $(i,j) \in [\lambda]$ is defined to be the difference $c(i,j):=j-i$. 
For $s \in \Z$, we define the $s$-\emph{charged content} of the node $(i,j)$ to be the integer $c(i,j)+s+1$.  
If now we look at Formula \eqref{ghl} for the generalised hook length of $(i,j)$ with respect to $(\lambda,\mu) \in \Pi \times \Pi$, we
 observe that $\lambda_i-i$ is the content of the rightmost box in the $i$-th row of the Young diagram of $\lambda$, while $j-\mu_j'$ is the content of the box at the bottom of the $j$-th column of the Young diagram of $\mu$; in order for the latter to hold even when the $j$-th column is empty, we may assume that the Young diagram of $\mu$ has  an imaginary $0$-th row with an infinite amount of boxes (see the definition of extended Young diagram in \cite{JKle}).
 This approach to the definition of the generalised hook length will serve us later on when we work with charged hook lengths.
\end{Rem}

\subsection{Specialisations and decomposition matrices}\label{sec-U}

  Let  $k$ be a field and let $\theta : R \to k$ 
   be a specialisation of $R$.  Set $\xi_i:=\theta (Q_i)$, for $i=0,1,\ldots,l-1$,  $u:=\theta (q)$ and ${\bf u}:=(\xi_0,\,\ldots,\,\xi_{l-1}\,;\,u)$. 
 Assume that the algebra  $k\mathcal{H}^{\bf u}_{n}$ is split. In \cite[Theorem 4.2]{CJ}, we have shown that the semisimplicity criterion for symmetric algebras given at the end of \S\ref{sec2.1} in combination with the form of the Schur elements given by Theorem \ref{canfreeform} allows us to recover Ariki's semisimplicity criterion for Ariki--Koike algebras, which is the following \cite[Main theorem]{Arsem}:
 
 \begin{Th}\label{Ariki's semisimplicity} The algebra
 $k \mathcal{H}^{\bf u}_n$ is  semisimple if and only if 
 $$\prod_{1\leq i \leq n}(1+u+\cdots+u^{i-1}) \prod_{0 \leq a <b \leq l-1}\,\,\prod_{-n<h<n}(u^h\xi_a-\xi_b) \neq 0.$$
 \end{Th}

In any case, we have a well-defined decomposition matrix
$D_\theta=([V^{\ulambda}:M])_{\ulambda\in \Pi^l (n),\,M\in \text{Irr} ( k  \mathcal{H}_n^{\bf u})}$.   There is a useful result 
 by Dipper and Mathas \cite{DiMa} which allows us to restrict ourselves to a very specific situation in order to study $D_\theta$. In order to do this, we set $\mathcal{U}:=\left\{0,1,\ldots ,{l-1}\right\}$ and we assume that we have a partition 
$$\mathcal{U}=\mathcal{U}_1 \sqcup   \mathcal{U}_2 \sqcup \ldots  \sqcup \mathcal{U}_t $$
which is the finest with respect to the property
$$\prod_{1\leq \alpha<\beta\leq t}\,\, \prod_{(a,b)\in \mathcal{U}_{\alpha}\times \mathcal{U}_{\beta}}\,\, \prod_{-n<h<n} (u^h \xi_a-\xi_b) \neq 0.$$ 
 For $i=1,\ldots,t$, write $\mathcal{U}_i:=\{a_{i,1},\ldots,a_{i,m_i}\}$ with $a_{i,1} < \cdots < a_{i,m_i}$. Whenever ${\bf f}=(f_0,\ldots,f_{l-1})$ is a sequence  indexed by $\mathcal{U}$, we will write ${\bf f}[i]$ for the sequence $(f_{a_{i,1}},\ldots,f_{a_{i,m_i}})$.

 \begin{Th}\label{Nmorita}
{\bf (The  Morita equivalence of Dipper and Mathas)}  For $i=1,\ldots,t$, we set ${\bf u_i}:=((\xi_0,\,\ldots,\,\xi_{l-1})[i];\,u)$. 
The algebra $k\mathcal{H}^{\bf u}_{n}$ is Morita equivalent to the algebra
$$\bigoplus_{\underset{n_1,\ldots ,n_t\geq 0}{n_1+\ldots +n_t=n}}  k \mathcal{H}^{\bf u_1}_{n_1}\otimes_k   k\mathcal{H}^{\bf u_2}_{n_2}\otimes_k  \ldots \otimes_k    k\mathcal{H}_{n_t}^{\bf u_t}.    $$
\end{Th}

\begin{Rem}
Recently, Rostam \cite{Salim} has produced an explicit isomorphism between $k\mathcal{H}^{\bf u}_{n}$ and
$$\bigoplus_{\underset{n_1,\ldots ,n_t\geq 0}{n_1+\ldots +n_t=n}} {\rm Mat}_{\frac{n!}{n_1!\dots n_t!}} \left(k \mathcal{H}^{\bf u_1}_{n_1}\otimes_k   k\mathcal{H}^{\bf u_2}_{n_2}\otimes_k  \ldots \otimes_k    k\mathcal{H}_{n_t}^{\bf u_t}\right),    $$
which implies the Morita equivalence of Dipper and Mathas.
\end{Rem}

If, for each $i=1,\ldots, t$, we take $M_i$ to be a simple $k \mathcal{H}^{\bf u_i}_{n_i}$-module, then, under the Morita equivalence of Dipper and Mathas, we get a  simple module  $M\in \operatorname{Irr} (k\mathcal{H}^{\bf u}_{n})$ 
 associated to $M_1\otimes\ldots  \otimes M_t$. For all $\ulambda\in \Pi^l (n)$, we have:
 \begin{equation}\label{decform}
[V^\ulambda:M]=[V^{\ulambda[1]}:M_1] \ldots [V^{\ulambda[t]}:M_t].
\end{equation}
Moreover, all the remaining entries of $D_\theta$ must be equal to $0$. 

Let us now consider the Brauer graph associated with $\theta$. Let $\ulambda, \umu \in \Pi^l (n)$. By \eqref{decform}, we have
$[V^\ulambda:M] \neq 0 \neq [V^\umu:M]$ if and only if
$[V^{\ulambda[i]}:M_i] \neq 0 \neq [V^{\umu[i]}:M_i]$ for all $i=1,\ldots,t$. This can obviously only happen if $\ulambda[i]$ and $\umu[i]$ have the same rank, say $n_i$, for all $i=1,\ldots,t$. We deduce that $V^\ulambda$ and $V^\umu$ are in the same $\theta$-block if and only if $V^{\ulambda[i]}$ and $V^{\umu[i]}$ are in the same block (of $k \mathcal{H}^{\bf u_i}_{n_i}$) for all $i=1,\ldots,t$. Sometimes, for simplicity, we say that $\ulambda$ and $\umu$ are in the same block to mean that $V^\ulambda$ and $V^\umu$ are in the same block.

Let $\mathcal{O}$ be the localisation of $R$ at ${\rm Ker}\theta$ and assume that $\mathcal{O}$ is a discrete valuation ring. Let $\Phi$ be a generator of the maximal ideal of $\mathcal{O}$. Let $\ulambda \in \Pi^{l}(n)$.
For $1\leq \alpha<\beta\leq t$ and ${(a,b)\in \mathcal{U}_{\alpha}\times \mathcal{U}_{\beta}}$, we have:
 $$ u^{h_{i,j}^{\lambda^{a},\lambda^{b}}} \xi_{a} \xi_{b}^{-1}-1\neq 0
 \quad \text{ for all } (i,j) \in [\lambda^a],
 $$
 whence
 $$ \nu_\Phi(q^{h_{i,j}^{\lambda^{a},\lambda^{b}}} Q_{a} Q_{b}^{-1}-1)= 0
 \quad \text{ for all } (i,j) \in [\lambda^a].$$
 Thanks to the form of the Schur elements of the Ariki--Koike algebras given by \eqref{claim}, we deduce the following formula for the calculation of the $\Phi$-defect of $V^{\ulambda}$:
 \begin{equation}\label{FormuladefectAK}
 \nu_\Phi(s_{\ulambda}({\bf q}))=\sum_{i=1}^t \nu_\Phi(s_{\ulambda[i]}({\bf q_i})),
 \end{equation}
where  ${\bf q_i}=((Q_0,\,\ldots,\,Q_{l-1})[i]\,;\,q)$.
That is, we have that the $\Phi$-defect of $V^{\ulambda}$ is the sum of the $\Phi$-defects of the
$V^{\ulambda[i]}$'s.

\subsection{Cyclotomic Ariki--Koike algebras}

Let  $\eta_{l}:={\rm exp}(2\pi i/l)$ and $K:=\mathbb{Q}(\eta_l)$. 
Let $y$ be an indeterminate, 
and let $\varphi: \mathbb{Z}_K[{\bf q},{\bf q}^{-1}] \rightarrow \mathbb{Z}_K[y,y^{-1}]$   be a cyclotomic specialisation of $\mathcal{H}^{\bf q}_{n}$ such that
\begin{eqnarray*}\label{cy}
&&\varphi(Q_i)=\eta_{l}^{i}y^{r_{i}} \textrm{ for}\ i=0,1,\ldots ,l-1, \\
&&\varphi(q)=y^{r} 
\end{eqnarray*}
where $(r_{0},\ldots ,r_{l-1},r)\in \mathbb{Z}^{l+1}$. The cyclotomic Ariki--Koike algebra $K(y) (\mathcal{H}^{\bf q}_{n})_\varphi$ is split semisimple.

Let now $\eta$ be a non-zero complex number, and let $\theta: \mathbb{Z}_{K}[y,y^{-1}] \rightarrow K(\eta), y \mapsto \eta$ be a specialisation of $\mathbb{Z}_{K}[y,y^{-1}]$ such that $\theta(\eta_l)$ is a primitive $l$-th root of unity (for simplicity, we may assume that $\theta$ is a $\Z_K$-algebra morphism, whence $\theta(\eta_l)=\eta_l$). Let $\Phi$ denote the minimal polynomial of $\eta$ over the field $K$. The aim of Section \ref{sec-AK} will be to prove the following result:

\begin{Th}\label{main-AK}
Let $\ulambda,\umu \in \Pi^l(n)$. If $V^{\ulambda}$ and $V^{\umu}$ are in the same $\theta$-block, then they have the same $\Phi$-defect, that is,
$\nu_\Phi(\varphi(s_{\ulambda}({\bf q})))=\nu_\Phi(\varphi(s_{\umu}({\bf q})))$.
\end{Th}

Set $\xi_i:=\eta_l^i \eta^{r_i} =\theta (\varphi(Q_i))$, for $i=0,1,\ldots,l-1$,  $u:=\eta^r=\theta (\varphi(q))$ and ${\bf u}:=(\xi_0,\,\ldots,\,\xi_{l-1}\,;\,u)$. 
It follows from Proposition \ref{Schur element cyclotomic} that the algebra $K(\eta)\mathcal{H}^{\bf u}_{n}$ is semisimple unless $\eta$ is a root of unity, which is true if and only if $u$ is a root of unity.

From now on, let us assume that $u$ is a primitive root of unity of order $e>0$, and set $k:=K(\eta)$. Using the notation of the previous subsection, we have that $k\mathcal{H}^{\bf u}_{n}$ is Morita equivalent to the algebra 
$$\bigoplus_{\underset{n_1,\ldots ,n_t\geq 0}{n_1+\ldots +n_t=n}}  k \mathcal{H}^{\bf u_1}_{n_1}\otimes_k   k\mathcal{H}^{\bf u_2}_{n_2}\otimes_k  \ldots \otimes_k    k\mathcal{H}_{n_t}^{\bf u_t}.    $$
For $i= 1,\ldots,t$, by definition of the sets $\mathcal{U}_i$, there exists $(s_{a_{i,1}},\ldots,s_{a_{i,{m_i}}}) \in \mathbb{Z}^{m_i}$ such that $\xi_{a_{i,j}}/\xi_{a_{i,1}}=u^{s_{a_{i,j}}}$ for all $j=1,\ldots,m_i$. 
In fact, for any $n_i \geq 0$, the algebra $k\mathcal{H}^{{\bf u}_i}_{n_i}$ is isomorphic to the algebra  $k\mathcal{H}^{\bf \widetilde{u}_i}_{n_i}$ where ${\bf \widetilde{u}}_i:=(u^{s_{a_{i,1}}},\ldots,u^{s_{a_{i,{m_i}}}};\,u)$. 

Let  $\ulambda \in \Pi^{l}(n)$ and 
$i \in \{1,\ldots,t\}$. For all $a, b \in \mathcal{U}_i$ and $(c,d) \in [\lambda^a]$, we have:
$$\theta(\varphi(q^{h_{c,d}^{\lambda^a,\lambda^{b}}}Q_aQ_b^{-1}-1))=
\theta(\eta_l^{a-b}y^{rh_{c,d}^{\lambda^a,\lambda^{b}}+r_a-r_b}-1)=
\eta_l^{a-b}\eta^{rh_{c,d}^{\lambda^a,\lambda^{b}}+r_a-r_b}-1=
u^{h_{c,d}^{\lambda^a,\lambda^{b}}+s_a-s_b}-1,$$
whence we obtain that
\begin{equation}\label{combo1}
\eta_l^{a-b}\eta^{rh_{c,d}^{\lambda^a,\lambda^{b}}+r_a-r_b}-1=0 \Leftrightarrow
h_{c,d}^{\lambda^a,\lambda^{b}}+s_a-s_b \equiv 0 \text{ mod } e.
\end{equation}
Set $M:=r h_{c,d}^{\lambda^a,\lambda^{b}}+r_a-r_b$. 
If $M \neq 0$, 
consider the following polynomial inside $K[y]$:
$$P(y):=\left\{ \begin{array}{ll}
\eta_l^{a-b}y^{M}-1& \text{if } M>0,\\ &\\
\eta_l^{b-a}y^{-M}-1 & \text{if } M<0.
\end{array}\right.$$
Since ${\rm gcd}(P(y),P'(y))=1$, all roots of $P(y)$ are simple.
Thus,
we have that $P(\eta)=0$ if and only if $\Phi$ divides $P(y)$ exactly once. That is,
\begin{equation}\label{combo2}
\eta_l^{a-b}\eta^M-1=0 \Leftrightarrow \nu_\Phi (\eta_l^{a-b}y^{M}-1)=1.
\end{equation}
If $M=0$ and $a \neq b$, then $\eta_l^{a-b}-1 \neq 0$ and
$\nu_\Phi(\eta_l^{a-b}-1)=0$. If $M=0$ and $a=b$, then we must have $r=0$, because the classical hook length of a node is always non-zero. However, in this case, we are covered by the case $u=e=1$ below. 
Combining \eqref{combo1} and \eqref{combo2} with \eqref{claim}, we deduce that $\nu_\Phi(\varphi(s_{\ulambda[i]}({\bf q_i})))$ is equal to the number of elements of the multiset 
$$H({\ulambda})_i:=
\left\{ \begin{array}{ll}
\{h_{c,d}^{\lambda^a,\lambda^{b}}+s_a-s_b \,|\, a,b \in \mathcal{U}_i,\,(c,d) \in [\lambda^a]  \} & \text{if } e>1,\\ &\\
\{h_{c,d}^{\lambda^a,\lambda^{b}}+s_a-s_b \,|\, a,b \in \mathcal{U}_i,\, a \neq b,\,(c,d) \in [\lambda^a]  \} & \text{if } e=1
\end{array}\right.$$
 that are congruent to $0$ modulo $e$. If $e=1$, this number is simply the cardinality of the multiset $H({\ulambda})_i$, which is equal to $|\ulambda[i] |(m_i-1)$. 

\begin{Rem}\label{difchoice}
The value of $\nu_\Phi(\varphi(s_{\ulambda[i]}({\bf q_i})))$  is not affected by our choice of $(s_{a_{i,1}},\ldots,s_{a_{i,{m_i}}})$. If  $(s'_{a_{i,1}},\ldots,s'_{a_{i,{m_i}}}) \in \Z^{m_i}$ also satisfies $\xi_{a_{i,j}}/\xi_{a_{i,1}}=u^{s'_{a_{i,j}}}$, then 
 ${s_{a_{i,j}}}\equiv {s'_{a_{i,j}}} \text{ mod } e$, for all $j=1,\ldots,m_i$.
\end{Rem}

We can now use Formula \eqref{FormuladefectAK} to calculate $\nu_\Phi(\varphi(s_{\ulambda}({\bf q})))$. We obtain:
\begin{equation}\label{defAK}
\nu_\Phi(\varphi(s_{\ulambda}({\bf q}))) = \sum_{i=1}^t \sharp\{h \in H({\ulambda})_i\,\,|\,\, h\equiv 0\text{ mod } e\}.
\end{equation}
Therefore,  in order to study the defect in Ariki--Koike algebras,  it is enough to focus on specialised Ariki--Koike algebras of the form $\mathcal{H}^{\bf u}_{n}$, where  ${\bf u}=(u^{s_{0}},\ldots ,u^{s_{l-1}}\;;\, u)$ for some $(s_{0},\ldots ,s_{l-1})\in \mathbb{Z}^{l}$.  In particular, in order to prove  Theorem \ref{main-AK}, it is enough to prove the following: for $\ulambda, \umu \in \Pi^l(n)$, we have
\begin{equation}\label{equivmain}
\ulambda, \umu \text{ are in the same block of  }
\mathcal{H}^{\bf u}_{n} \Rightarrow 
\sharp\{h \in H({\ulambda})\,\,|\,\, h\equiv 0\text{ mod } e\}=\sharp\{h \in H({\umu})\,\,|\,\, h\equiv 0\text{ mod } e\},
\end{equation}
where
$$H({\ulambda}):=
\left\{ \begin{array}{ll}
\{h_{c,d}^{\lambda^a,\lambda^{b}}+s_a-s_b \,|\, a,b \in \{0,1,\ldots,l-1\},\,(c,d) \in [\lambda^a]  \} \} & \text{if } e>1,\\ &\\
\{h_{c,d}^{\lambda^a,\lambda^{b}}+s_a-s_b \,|\, a,b \in \{0,1,\ldots,l-1\},\,a\neq b,\,(c,d) \in [\lambda^a]  \} & \text{if } e=1.
\end{array}\right.$$
For $e=1$, \eqref{equivmain} is automatically true, since $\sharp\{h \in H({\ulambda})\,\,|\,\, h\equiv 0\text{ mod } e\} =\sharp H(\ulambda)= n(l-1)$ for all $\ulambda \in \Pi^l(n)$. Therefore, in the rest of the paper, we are only going to consider the case $e>1$.

\begin{Rem}\label{good Ariki-Koike}
Let us choose $(s_{0},\ldots ,s_{l-1})\in \mathbb{Z}^{l}$ so that all elements of the multiset $H({\ulambda})$  are non-zero, for all $\ulambda \in \Pi^l(n)$.  If we now consider the $\Z$-algebra morphism $\varphi: \mathbb{Z}[{\bf q},{\bf q}^{-1}] \rightarrow \mathbb{Z}[y,y^{-1}]$   such that
\begin{eqnarray*}\label{cy0}
&&\varphi(Q_i)=y^{s_{i}} \textrm{ for}\ i=0,1,\ldots ,l-1, \\
&&\varphi(q)=y, 
\end{eqnarray*}
then the algebra $\Q(y) (\mathcal{H}^{\bf q}_{n})_\varphi$ is split semisimple and the $e$-defect of $V^{\ulambda}$  is equal to the number of elements of $H({\ulambda})$ that are congruent to $0$ modulo $e$, for all $\ulambda \in \Pi^l(n)$.
\end{Rem}

\begin{Rem}
One could ask, especially after seeing Formula \eqref{defAK}, whether we can replace $\Phi$-defect with $e$-defect when $\eta=u$, for any cyclotomic Ariki--Koike algebra. We saw that it is the case when the Ariki--Koike algebra is as in the previous remark, and it is also the case when we deal with real reflection groups. However, in the complex case,   the $e$-defect is not always well-defined, because $K$-cyclotomic polynomials may not be the same as the cyclotomic polynomials over $\Q$. 

Let $l=3$ and consider the  cyclotomic Ariki--Koike algebra $(\mathcal{H}_2^{\bf q})_\varphi$, where 
$$\varphi(Q_0)=1,\,\,\,\varphi(Q_1)=\eta_3,\,\,\,\varphi(Q_2)=\eta_3^2 y \quad \text{and}
\quad \varphi(q)=y.$$
We have $K=\Q(\eta_3)$ and, over $K$, the minimal polynomials of the third roots of unity are of degree $1$.
For $\ulambda=(2,\emptyset,\emptyset)$, we have
$$\varphi(s_{\ulambda}({\bf q}))=-3\eta_3^2 y^{-1} (y+1)(y-\eta_3)^2,$$
and so $\nu_{y-\eta_3}(\varphi(s_{\ulambda}({\bf q})))=2$, while
$\nu_{y-\eta_3^2}(\varphi(s_{\ulambda}({\bf q})))=0$.  The $2$-defect of $V^{\ulambda}$ is equal to $1$.

Nevertheless, thanks to the work of Brundan and Kleshchev on cyclotomic quiver Hecke algebras \cite{BrKl}, we know that specialising $y$ to roots of unity of the same order yields isomorphic Ariki--Koike algebras. Therefore, the representation theory remains the same, even though the labelling of the representations might change.
In this particular example, for $\umu=(\emptyset,2,\emptyset)$, we have
$$\varphi(s_{\umu}({\bf q}))=-3\eta_3 y^{-1} (y+1)(y-\eta_3^2)^2,$$
and so the roles of $\ulambda$ and $\umu$ interchange, depending on whether $y \mapsto \eta_3$ or $y \mapsto \eta_3^2$.
\end{Rem}

\subsection{Charged  hook lengths and abaci}
We begin with several classical definitions. 
An {\it abacus}   is a subset $L$ of $\mathbb{Z}$
  such that $-i\in L$ and $i\notin L$ for all $i$ large enough.
    In a less formal way, each $i\in L$ corresponds to the position of a black bead
on the horizontal abacus  which is full of  black beads on the left and empty on the right. 
 One can associate to $\lambda \in \Pi$ and $s\in \mathbb{Z}$ an abacus $L_s (\lambda)$ 
  such that $k\in L_s (\lambda)$ if and only if there exists $j\in \mathbb{Z}_{>0}$ such that $k=\lambda_j-j+s+1$ (note that 
  $\lambda$  is assumed to have an
infinite number of zero parts). The elements of  $L_s (\lambda)$ 
 are called the \emph{$s$-charged $\beta$-numbers of $\lambda$}. Given an abacus $L$, one can easily find the unique partition $\lambda$ and the integer $s\in \mathbb{Z}$
 such that $L_s (\lambda)=L$. Indeed, each part of $\lambda$ corresponds to a black bead of the abacus and is equal to the number of empty positions at its left, while
 $s$ is the integer labelling the position obtained by the rightmost black bead  after sliding all the black beads in the abacus to the left. 
If  now we fix $m\in \mathbb{Z}_{>0}$ such that $\lambda_{m+s}=0$, we denote by $X[\lambda,m]$  the tuple of $s$-charged $\beta$-numbers 
$(\beta_1,\beta_2,\ldots,\beta_{m+s})$, where $\beta_j:=\lambda_j-j+s+1$ for all
$j \in \Z_{>0}$. Using the terminology of Remark \ref{tobeusedlater}, $\beta_j$ is the $s$-charged content of the rightmost box in the $j$-th row of the Young diagram of $\lambda$ (again, in order for this to hold even when the $j$-th row is empty, we may assume that the Young diagram of $\lambda$ has  an imaginary $0$-th column with an infinite amount of boxes). 
Moreover, note that $\beta_1>\beta_2>\dots>\beta_{m+s}$ and that $\beta_{m+s}=-m+1$.

\begin{exa}
Let us take the partition $\lambda=(5,4,2,1,1)$ and $s=0$.  The associated abacus $L_{0}(\lambda)$ may be represented as follows, where the positions  to the right 
 of the dashed vertical line are labelled by the non-negative  integers:

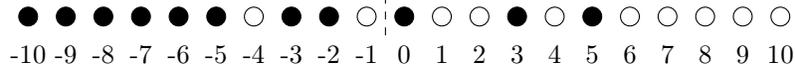
\[
\begin{tikzpicture}[scale=0.5, bb/.style={draw,circle,fill,minimum size=2.5mm,inner sep=0pt,outer sep=0pt}, wb/.style={draw,circle,fill=white,minimum size=2.5mm,inner sep=0pt,outer sep=0pt}]
\node [wb] at (11,2) {};
\node [wb] at (10,2) {};
\node [wb] at (9,2) {};
\node [wb] at (8,2) {};
\node [wb] at (7,2) {};
\node [bb] at (6,2) {};
\node [wb] at (5,2) {};
\node [bb] at (4,2) {};
\node [wb] at (3,2) {};
\node [wb] at (2,2) {};
\node [bb] at (1,2) {};
\node [wb] at (0,2) {};
\node [bb] at (-1,2) {};
\node [bb] at (-2,2) {};
\node [wb] at (-3,2) {};
\node [bb] at (-4,2) {};
\node [bb] at (-5,2) {};
\node [bb] at (-6,2) {};
\node [bb] at (-7,2) {};
\node [bb] at (-8,2) {};
\node [bb] at (-9,2) {};
\draw[dashed](0.5,1.5)--node[]{}(0.5,2.5);

	\node [] at (11,1) {10};
	\node [] at (10,1) {9};
	\node [] at (9,1) {8};
	\node [] at (8,1) {7};
	\node [] at (7,1) {6};
	\node [] at (6,1) {5};
	\node [] at (5,1) {4};
	\node [] at (4,1) {3};
	\node [] at (3,1) {2};
	\node [] at (2,1) {1};
	\node [] at (1,1) {0};
	\node [] at (0,1) {-1};
	\node [] at (-1,1) {-2};
	\node [] at (-2,1) {-3};
	\node [] at (-3,1) {-4};
	\node [] at (-4,1) {-5};
	\node [] at (-5,1) {-6};
	\node [] at (-6,1) {-7};
	\node [] at (-7,1) {-8};
	\node [] at (-8,1) {-9};
	\node [] at (-9,1) {-10};
\end{tikzpicture}
\]
We have
$X[\lambda,6]=(5,3,0,-2,-3,-5).
$
\end{exa}
Note that the abacus may be recovered from the datum of $X[\lambda,m]$ by setting  black beads in the positions given in $X[\lambda,m]$ and in all the positions  to the left of $\beta_{m+s}$. 

Now, let $\ulambda \in \Pi^l$ and ${\bf s}=(s_0,\ldots,s_{l-1})\in \mathbb{Z}^l$;
we refer to ${\bf s}$ as a \emph{multicharge}. One can associate to $\ulambda$ and ${\bf s}$ an  $l$-{\emph{abacus}} defined as the $l$-tuple $(L_{s_0} (\lambda^0),\ldots,L_{s_{l-1}} (\lambda^{l-1}))$. This $l$-abacus is pictured 
by stacking the abaci so that $L_{s_{i}} (\lambda^{i})$ is right on top of $L_{s_{i-1}} (\lambda^{i-1})$, for $i=1,\ldots,l-1$, and the positions are aligned vertically.
 If  we fix $m\in \mathbb{Z}_{>0}$ 
  such that $\lambda^i_{m+s_{i}}=0$ for all $i=0,1,\ldots,l-1$, then we denote by $X[\ulambda,m]$ the $l$-tuple $(X[\lambda^0,m],\ldots,X[\lambda^{l-1},m])$. 
  We will simply write $X^i$ for $X[\lambda^i,m]$ and we will denote by
  $(\beta^i_1,\ldots,\beta^i_{m+s_i})$ the elements of $X^i$, for all $i=0,1,\ldots,l-1$.

  \begin{exa}\label{firstabacus}
  Let $l=3$, $\ulambda=(2,1,1.1)$ and  ${\bf s}=(0,1,2)$. 
    The associated  $l$-abacus can be represented as follows:
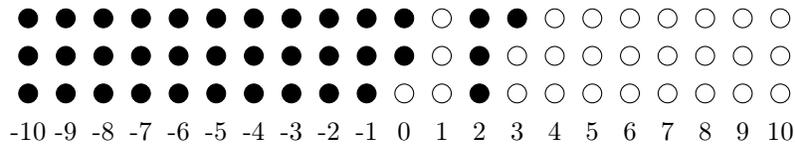
\begin{center}
\begin{tikzpicture}[scale=0.5, bb/.style={draw,circle,fill,minimum size=2.5mm,inner sep=0pt,outer sep=0pt}, wb/.style={draw,circle,fill=white,minimum size=2.5mm,inner sep=0pt,outer sep=0pt}]
	
	\node [] at (11,-1) {10};
	\node [] at (10,-1) {9};
	\node [] at (9,-1) {8};
	\node [] at (8,-1) {7};
	\node [] at (7,-1) {6};
	\node [] at (6,-1) {5};
	\node [] at (5,-1) {4};
	\node [] at (4,-1) {3};
	\node [] at (3,-1) {2};
	\node [] at (2,-1) {1};
	\node [] at (1,-1) {0};
	\node [] at (0,-1) {-1};
	\node [] at (-1,-1) {-2};
	\node [] at (-2,-1) {-3};
	\node [] at (-3,-1) {-4};
	\node [] at (-4,-1) {-5};
	\node [] at (-5,-1) {-6};
	\node [] at (-6,-1) {-7};
	\node [] at (-7,-1) {-8};
	\node [] at (-8,-1) {-9};
	\node [] at (-9,-1) {-10};

	\node [wb] at (11,0) {};
	\node [wb] at (10,0) {};
	\node [wb] at (9,0) {};
	\node [wb] at (8,0) {};
	\node [wb] at (7,0) {};
	\node [wb] at (6,0) {};
	\node [wb] at (5,0) {};
	\node [wb] at (4,0) {};
	\node [bb] at (3,0) {};
	\node [wb] at (2,0) {};
	\node [wb] at (1,0) {};
	\node [bb] at (0,0) {};
	\node [bb] at (-1,0) {};
	\node [bb] at (-2,0) {};
	\node [bb] at (-3,0) {};
	\node [bb] at (-4,0) {};
	\node [bb] at (-5,0) {};
	\node [bb] at (-6,0) {};
	\node [bb] at (-7,0) {};
	\node [bb] at (-8,0) {};
	\node [bb] at (-9,0) {};
	
	\node [wb] at (11,1) {};
	\node [wb] at (10,1) {};
	\node [wb] at (9,1) {};
	\node [wb] at (8,1) {};
	\node [wb] at (7,1) {};
	\node [wb] at (6,1) {};
	\node [wb] at (5,1) {};
	\node [wb] at (4,1) {};
	\node [bb] at (3,1) {};
	\node [wb] at (2,1) {};
	\node [bb] at (1,1) {};
	\node [bb] at (0,1) {};
	\node [bb] at (-1,1) {};
	\node [bb] at (-2,1) {};
	\node [bb] at (-3,1) {};
	\node [bb] at (-4,1) {};
	\node [bb] at (-5,1) {};
	\node [bb] at (-6,1) {};
	\node [bb] at (-7,1) {};
	\node [bb] at (-8,1) {};
	\node [bb] at (-9,1) {};
	
	\node [wb] at (11,2) {};
	\node [wb] at (10,2) {};
	\node [wb] at (9,2) {};
	\node [wb] at (8,2) {};
	\node [wb] at (7,2) {};
	\node [wb] at (6,2) {};
	\node [wb] at (5,2) {};
	\node [bb] at (4,2) {};
	\node [bb] at (3,2) {};
	\node [wb] at (2,2) {};
	\node [bb] at (1,2) {};
	\node [bb] at (0,2) {};
	\node [bb] at (-1,2) {};
	\node [bb] at (-2,2) {};
	\node [bb] at (-3,2) {};
	\node [bb] at (-4,2) {};
	\node [bb] at (-5,2) {};
	\node [bb] at (-6,2) {};
	\node [bb] at (-7,2) {};
	\node [bb] at (-8,2) {};
	\node [bb] at (-9,2) {};
	\end{tikzpicture}
\end{center}
We have   
$X[\ulambda,3]=((2,-1,-2),(2,0,-1,-2),(3,2,0,-1,-2)).$
  \end{exa}
  
  Obviously, the $l$-abacus can be easily recovered from the datum 
   of $X[\ulambda,m]$.  From now on, we will often switch from one representation using abaci to the one  using $\beta$-numbers. 

Let $a,b \in \{0,1,\ldots,l-1\}$ and $(c,d) \in  [\lambda^a]$. We define the ${\bf s}$-\emph{charged hook length of $(c,d)$ with respect to} $(\lambda^a,\lambda^b)$ to be the integer
$$ch_{c,d}^{\lambda^a,\lambda^{b}}:=h_{c,d}^{\lambda^a,\lambda^{b}}+s_a-s_b.$$
 We have
$$ch_{c,d}^{\lambda^a,\lambda^{b}}= \lambda^a_c-c+{\lambda^b}'_d-d+1+s_a-s_b
=(\lambda^a_c-c+s_a+1)-(d-{\lambda^b}'_d+s_b+1)+1
$$
We  observe that $\lambda^a_c-c+s_a+1=\beta^a_c$, while 
$d-{\lambda^b}'_d+s_b+1$ is the $s_b$-charged content of the box at the bottom of the $d$-th column of the Young diagram of $\lambda^b$ (cf.~Remark \ref{tobeusedlater}).
The latter coincides with the   leftmost position in the component $b$ of the abacus that has exactly $d$ empty positions at its left. In fact,
$d-{\lambda^b}'_d+s_b$ is the $d$-th empty position in the component $b$ of the abacus, reading from left to right.
Therefore, we can use the  set of $\beta$-numbers and the associated abacus decomposition  of our $l$-partition in order to easily compute the elements of the multiset
$$H(\ulambda):=\{ch_{c,d}^{\lambda^a,\lambda^{b}}\,|\,a,b \in \{0,1,\ldots,l-1\},\,(c,d) \in [\lambda^a]  \}$$
as follows:
\begin{itemize}
\item Let $\beta$ be the position of a black bead. \smallbreak
\item Let $\delta(\beta)$ be the number of empty positions to the left of $\beta$.
\smallbreak
\item For $b=0,1\ldots,l-1$ and $d=1,\ldots,\delta(\beta)$, let $y^b_d$ be the $d$-th empty position in the component $b$ of the abacus, reading from left to right.\smallbreak
\item  For $b=0,1\ldots,l-1$ , define $H^b(\beta)$ to be the set
$\{\beta - y^b_d\,|\,d=1,\ldots,\delta(\beta)\}$.
Note that:  \smallbreak
\begin{itemize}
\item if $\delta(\beta)=0$, then $H^b(\beta)=\emptyset$;
\item if $\delta(\beta)\neq 0$ and $\beta=\beta_c^a$, then
$\beta - y^b_d=ch_{c,d}^{\lambda^a,\lambda^{b}}$. 
\end{itemize}

\item Denote by $H(\beta)$ the multiset defined by the union $\bigcup_{b=0}^{l-1} H^b(\beta)$. \smallbreak
\item The multiset $H(\ulambda)$ is the union of the multisets $H(\beta)$, where $\beta$ runs over the positions of black beads
 (that is, the $\beta$-numbers $\beta^a_c$). It is enough to consider the black beads that have at least one empty position at their left  (that is, $\lambda^a_c \neq 0$), since otherwise $H(\beta)=\emptyset$.
\end{itemize}

\begin{exa}\label{second abacus}
Let $l=2$, $\ulambda=(3.1,2.1.1)$ and ${\bf s}=(0,2)$. We write the associated $2$-abacus as follows:\smallbreak

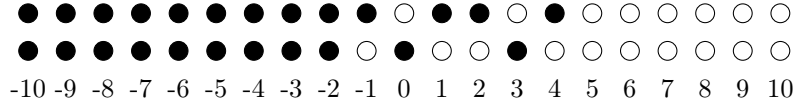
\begin{center}
\begin{tikzpicture}[scale=0.5, bb/.style={draw,circle,fill,minimum size=2.5mm,inner sep=0pt,outer sep=0pt}, wb/.style={draw,circle,fill=white,minimum size=2.5mm,inner sep=0pt,outer sep=0pt}]
	
	\node [] at (11,-1) {10};
	\node [] at (10,-1) {9};
	\node [] at (9,-1) {8};
	\node [] at (8,-1) {7};
	\node [] at (7,-1) {6};
	\node [] at (6,-1) {5};
	\node [] at (5,-1) {4};
	\node [] at (4,-1) {3};
	\node [] at (3,-1) {2};
	\node [] at (2,-1) {1};
	\node [] at (1,-1) {0};
	\node [] at (0,-1) {-1};
	\node [] at (-1,-1) {-2};
	\node [] at (-2,-1) {-3};
	\node [] at (-3,-1) {-4};
	\node [] at (-4,-1) {-5};
	\node [] at (-5,-1) {-6};
	\node [] at (-6,-1) {-7};
	\node [] at (-7,-1) {-8};
	\node [] at (-8,-1) {-9};
	\node [] at (-9,-1) {-10};

	\node [wb] at (11,0) {};
	\node [wb] at (10,0) {};
	\node [wb] at (9,0) {};
	\node [wb] at (8,0) {};
	\node [wb] at (7,0) {};
	\node [wb] at (6,0) {};
	\node [wb] at (5,0) {};
	\node [bb] at (4,0) {};
	\node [wb] at (3,0) {};
	\node [wb] at (2,0) {};
	\node [bb] at (1,0) {};
	\node [wb] at (0,0) {};
	\node [bb] at (-1,0) {};
	\node [bb] at (-2,0) {};
	\node [bb] at (-3,0) {};
	\node [bb] at (-4,0) {};
	\node [bb] at (-5,0) {};
	\node [bb] at (-6,0) {};
	\node [bb] at (-7,0) {};
	\node [bb] at (-8,0) {};
	\node [bb] at (-9,0) {};
	
	\node [wb] at (11,1) {};
	\node [wb] at (10,1) {};
	\node [wb] at (9,1) {};
	\node [wb] at (8,1) {};
	\node [wb] at (7,1) {};
	\node [wb] at (6,1) {};
	\node [bb] at (5,1) {};
	\node [wb] at (4,1) {};
	\node [bb] at (3,1) {};
	\node [bb] at (2,1) {};
	\node [wb] at (1,1) {};
	\node [bb] at (0,1) {};
	\node [bb] at (-1,1) {};
	\node [bb] at (-2,1) {};
	\node [bb] at (-3,1) {};
	\node [bb] at (-4,1) {};
	\node [bb] at (-5,1) {};
	\node [bb] at (-6,1) {};
	\node [bb] at (-7,1) {};
	\node [bb] at (-8,1) {};
	\node [bb] at (-9,1) {};

	\end{tikzpicture}
\end{center}

We begin with $\beta^0_1=3$. We observe that $\beta^0_1$ has $3$ empty positions to its left, and so $\delta(\beta^0_1)=3$.  
We have
$$y^0_1=-1,\,\, y^0_2=1, \,\, y^0_3=2, \,\, y^1_1=0, \,\,  y^1_2=3, \,\, y^1_3=5$$
ans so
$$H(\beta^0_1)= \{\beta^0_1-y^b_1, 
\beta^0_1-y^b_2, \beta^0_1-y^b_3  \,|\, b=0,1\}
 =\{3+1, 3-1, 3-2, 3-0, 3-3,3-5\}=\{4,2,1,3,0,-2\}$$
 
One may find it more convenient to compute  this set  the other way round.   We take the leftmost empty position  in the component $0$, labelled by $-1$, and the leftmost empty position  in the component $1$, labelled by $0$.

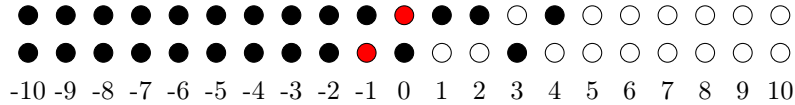
\begin{center}
\begin{tikzpicture}[scale=0.5, rb/.style={draw,circle,fill=red,minimum size=2.5mm,inner sep=0pt,outer sep=0pt}, bb/.style={draw,circle,fill,minimum size=2.5mm,inner sep=0pt,outer sep=0pt}, wb/.style={draw,circle,fill=white,minimum size=2.5mm,inner sep=0pt,outer sep=0pt}]
	
	\node [] at (11,-1) {10};
	\node [] at (10,-1) {9};
	\node [] at (9,-1) {8};
	\node [] at (8,-1) {7};
	\node [] at (7,-1) {6};
	\node [] at (6,-1) {5};
	\node [] at (5,-1) {4};
	\node [] at (4,-1) {3};
	\node [] at (3,-1) {2};
	\node [] at (2,-1) {1};
	\node [] at (1,-1) {0};
	\node [] at (0,-1) {-1};
	\node [] at (-1,-1) {-2};
	\node [] at (-2,-1) {-3};
	\node [] at (-3,-1) {-4};
	\node [] at (-4,-1) {-5};
	\node [] at (-5,-1) {-6};
	\node [] at (-6,-1) {-7};
	\node [] at (-7,-1) {-8};
	\node [] at (-8,-1) {-9};
	\node [] at (-9,-1) {-10};

	\node [wb] at (11,0) {};
	\node [wb] at (10,0) {};
	\node [wb] at (9,0) {};
	\node [wb] at (8,0) {};
	\node [wb] at (7,0) {};
	\node [wb] at (6,0) {};
	\node [wb] at (5,0) {};
	\node [bb] at (4,0) {};
	\node [wb] at (3,0) {};
	\node [wb] at (2,0) {};
	\node [bb] at (1,0) {};
	\node [rb] at (0,0) {};
	\node [bb] at (-1,0) {};
	\node [bb] at (-2,0) {};
	\node [bb] at (-3,0) {};
	\node [bb] at (-4,0) {};
	\node [bb] at (-5,0) {};
	\node [bb] at (-6,0) {};
	\node [bb] at (-7,0) {};
	\node [bb] at (-8,0) {};
	\node [bb] at (-9,0) {};
	
	\node [wb] at (11,1) {};
	\node [wb] at (10,1) {};
	\node [wb] at (9,1) {};
	\node [wb] at (8,1) {};
	\node [wb] at (7,1) {};
	\node [wb] at (6,1) {};
	\node [bb] at (5,1) {};
	\node [wb] at (4,1) {};
	\node [bb] at (3,1) {};
	\node [bb] at (2,1) {};
	\node [rb] at (1,1) {};
	\node [bb] at (0,1) {};
	\node [bb] at (-1,1) {};
	\node [bb] at (-2,1) {};
	\node [bb] at (-3,1) {};
	\node [bb] at (-4,1) {};
	\node [bb] at (-5,1) {};
	\node [bb] at (-6,1) {};
	\node [bb] at (-7,1) {};
	\node [bb] at (-8,1) {};
	\node [bb] at (-9,1) {};

	\end{tikzpicture}
\end{center}
So at this stage, we have $\{\beta^0_1-y_1^0,\beta^0_1-y_1^1\}=\{3+1=4,3-0=3\}$. We then consider the second leftmost empty position in each of the two components. \smallbreak

\begin{center}
\begin{tikzpicture}[scale=0.5, rb/.style={draw,circle,fill=red,minimum size=2.5mm,inner sep=0pt,outer sep=0pt}, bb/.style={draw,circle,fill,minimum size=2.5mm,inner sep=0pt,outer sep=0pt}, wb/.style={draw,circle,fill=white,minimum size=2.5mm,inner sep=0pt,outer sep=0pt}]
	
	\node [] at (11,-1) {10};
	\node [] at (10,-1) {9};
	\node [] at (9,-1) {8};
	\node [] at (8,-1) {7};
	\node [] at (7,-1) {6};
	\node [] at (6,-1) {5};
	\node [] at (5,-1) {4};
	\node [] at (4,-1) {3};
	\node [] at (3,-1) {2};
	\node [] at (2,-1) {1};
	\node [] at (1,-1) {0};
	\node [] at (0,-1) {-1};
	\node [] at (-1,-1) {-2};
	\node [] at (-2,-1) {-3};
	\node [] at (-3,-1) {-4};
	\node [] at (-4,-1) {-5};
	\node [] at (-5,-1) {-6};
	\node [] at (-6,-1) {-7};
	\node [] at (-7,-1) {-8};
	\node [] at (-8,-1) {-9};
	\node [] at (-9,-1) {-10};

	\node [wb] at (11,0) {};
	\node [wb] at (10,0) {};
	\node [wb] at (9,0) {};
	\node [wb] at (8,0) {};
	\node [wb] at (7,0) {};
	\node [wb] at (6,0) {};
	\node [wb] at (5,0) {};
	\node [bb] at (4,0) {};
	\node [wb] at (3,0) {};
	\node [rb] at (2,0) {};
	\node [bb] at (1,0) {};
	\node [wb] at (0,0) {};
	\node [bb] at (-1,0) {};
	\node [bb] at (-2,0) {};
	\node [bb] at (-3,0) {};
	\node [bb] at (-4,0) {};
	\node [bb] at (-5,0) {};
	\node [bb] at (-6,0) {};
	\node [bb] at (-7,0) {};
	\node [bb] at (-8,0) {};
	\node [bb] at (-9,0) {};
	
	\node [wb] at (11,1) {};
	\node [wb] at (10,1) {};
	\node [wb] at (9,1) {};
	\node [wb] at (8,1) {};
	\node [wb] at (7,1) {};
	\node [wb] at (6,1) {};
	\node [bb] at (5,1) {};
	\node [rb] at (4,1) {};
	\node [bb] at (3,1) {};
	\node [bb] at (2,1) {};
	\node [wb] at (1,1) {};
	\node [bb] at (0,1) {};
	\node [bb] at (-1,1) {};
	\node [bb] at (-2,1) {};
	\node [bb] at (-3,1) {};
	\node [bb] at (-4,1) {};
	\node [bb] at (-5,1) {};
	\node [bb] at (-6,1) {};
	\node [bb] at (-7,1) {};
	\node [bb] at (-8,1) {};
	\node [bb] at (-9,1) {};

	\end{tikzpicture}
\end{center}
We obtain $\{\beta^0_1-y_2^0,\beta^0_1-y_2^1\}=\{3-1=2,3-3=0\}$. We finally do the same for the third leftmost empty positions.
\begin{center}
\begin{tikzpicture}[scale=0.5, rb/.style={draw,circle,fill=red,minimum size=2.5mm,inner sep=0pt,outer sep=0pt}, bb/.style={draw,circle,fill,minimum size=2.5mm,inner sep=0pt,outer sep=0pt}, wb/.style={draw,circle,fill=white,minimum size=2.5mm,inner sep=0pt,outer sep=0pt}]
	
	\node [] at (11,-1) {10};
	\node [] at (10,-1) {9};
	\node [] at (9,-1) {8};
	\node [] at (8,-1) {7};
	\node [] at (7,-1) {6};
	\node [] at (6,-1) {5};
	\node [] at (5,-1) {4};
	\node [] at (4,-1) {3};
	\node [] at (3,-1) {2};
	\node [] at (2,-1) {1};
	\node [] at (1,-1) {0};
	\node [] at (0,-1) {-1};
	\node [] at (-1,-1) {-2};
	\node [] at (-2,-1) {-3};
	\node [] at (-3,-1) {-4};
	\node [] at (-4,-1) {-5};
	\node [] at (-5,-1) {-6};
	\node [] at (-6,-1) {-7};
	\node [] at (-7,-1) {-8};
	\node [] at (-8,-1) {-9};
	\node [] at (-9,-1) {-10};

	\node [wb] at (11,0) {};
	\node [wb] at (10,0) {};
	\node [wb] at (9,0) {};
	\node [wb] at (8,0) {};
	\node [wb] at (7,0) {};
	\node [wb] at (6,0) {};
	\node [wb] at (5,0) {};
	\node [bb] at (4,0) {};
	\node [rb] at (3,0) {};
	\node [wb] at (2,0) {};
	\node [bb] at (1,0) {};
	\node [wb] at (0,0) {};
	\node [bb] at (-1,0) {};
	\node [bb] at (-2,0) {};
	\node [bb] at (-3,0) {};
	\node [bb] at (-4,0) {};
	\node [bb] at (-5,0) {};
	\node [bb] at (-6,0) {};
	\node [bb] at (-7,0) {};
	\node [bb] at (-8,0) {};
	\node [bb] at (-9,0) {};
	
	\node [wb] at (11,1) {};
	\node [wb] at (10,1) {};
	\node [wb] at (9,1) {};
	\node [wb] at (8,1) {};
	\node [wb] at (7,1) {};
	\node [rb] at (6,1) {};
	\node [bb] at (5,1) {};
	\node [wb] at (4,1) {};
	\node [bb] at (3,1) {};
	\node [bb] at (2,1) {};
	\node [wb] at (1,1) {};
	\node [bb] at (0,1) {};
	\node [bb] at (-1,1) {};
	\node [bb] at (-2,1) {};
	\node [bb] at (-3,1) {};
	\node [bb] at (-4,1) {};
	\node [bb] at (-5,1) {};
	\node [bb] at (-6,1) {};
	\node [bb] at (-7,1) {};
	\node [bb] at (-8,1) {};
	\node [bb] at (-9,1) {};

	\end{tikzpicture}
\end{center}
We obtain $\{\beta^0_1-y_3^0,\beta^0_1-y_3^1\}=\{3-2=1,3-5=-2\}$ and so
$H(\beta^0_1)=\{4,3,2,0,1,-2\}$. 

Continuing this way, we get:

\begin{itemize}
 \item
 $H(\beta^0_2)=\{\beta^0_2-y^b_1 \,|\, b=0,1\}=
\{0+1,0-0\}=\{1,0\}$,\smallbreak

\item
$H(\beta^1_1)=\{\beta^1_1-y^b_1,  \beta^1_1-y^b_2 \,|\, b=0,1\}=\{4+1,4-1,4-0,4-3\}=\{5,3,4,1\}$,\smallbreak

\item
$H(\beta^1_2)=\{\beta^1_2-y^b_1 \,|\, b=0,1\}=\{2+1,2-0\}=\{3,2\}$,\smallbreak

\item
$H(\beta^1_3)=\{\beta^1_3-y^b_1 \,|\, b=0,1\}=
\{1+1,1-0\}=\{2,1\}.$\smallbreak
\end{itemize}
We obtain
 $${H} (\ulambda)=\{-2,0,0,1,1,1,1,2,2,2,3,3,3,4,4,5\}.$$
\end{exa}

\begin{exa}\label{firstabacusprime}
Let us also work out Example \ref{firstabacus}. We have
$$y^0_1=0,\,\, y^0_2=1, \,\, y^1_1=1, \,\, y^1_2=3, \,\,  y^2_1=1, \,\, y^2_2=4$$
and so
\begin{itemize}
\item $H(\beta^0_1)= \{\beta^0_1-y^b_1, 
\beta^0_1-y^b_2 \,|\, b=0,1,2\}
 =\{2-0, 2-1, 2-1, 2-3, 2-1,2-4\}=\{2,1,1,-1,1,-2\}$,\smallbreak
 \item
 $H(\beta^1_1)=\{\beta^1_1-y^b_1 \,|\, b=0,1,2\}=
\{2-0, 2-1, 2-1\}=\{2,1,1\}$,\smallbreak

\item
$H(\beta^2_1)=\{\beta^2_1-y^b_1 \,|\, b=0,1,2\}=
\{3-0, 3-1, 3-1\}=\{3,2,2\}$,\smallbreak

\item
$H(\beta^2_2)=\{\beta^2_2-y^b_1 \,|\, b=0,1,2\}=\{2-0, 2-1, 2-1\}=\{2,1,1\}$.\smallbreak

\end{itemize}
We obtain
 $${H} (\ulambda)=\{-2,-1,1,1,1,1,1,1,1,2,2,2,2,2,3\}.$$
\end{exa}

\subsection{Residue classes of charged hook lengths}
Let $\ulambda \in \Pi^l$, ${\bf s}=(s_0,\ldots,s_{l-1})\in \mathbb{Z}^l$ and $e \in  \mathbb{Z}_{>1}$.  We consider the $l$-abacus  associated to the pair $(\ulambda,{\bf s})$ and the set of $\beta$-numbers $X:=X[\ulambda,m]=(X^0,\ldots,X^{l-1})$. Recall that
 $\min(X^c)=1-m$ for all $c=0,\ldots,l-1$.
The aim of this subsection is to show how one can 
use $X$ to compute the number of charged  hook lengths congruent to $0$ modulo $e$. In order to do this, we will assume that ${\bf s}$ belongs to
 $$\mathcal{A}_e^l=\{ (s_0,\ldots,s_{l-1})\in \mathbb{Z}^l \ |\ s_0\leq s_1 \leq \ldots\leq s_{l-1} \leq s_0+e\},$$
 since permuting the entries of the multicharge or taking other representatives of their congruence classes modulo $e$ does not affect the Ariki--Koike algebra.

   We use the notation of the previous subsection for the description of the elements of $H(\ulambda)$. We have
\begin{equation}\label{general}
H(\ulambda) = \bigcup_{c_1=0}^{l-1} \bigcup_{c_2=0}^{l-1} 
\bigcup_{x \in {X}^{c_1}}
 H^{c_2}(x),
 \end{equation}
where $ H^{c_2}(x)= \{ x-y_d^{c_2} \,|\,d=1,\ldots,\delta(x)\}$. Note that  $ H^{c_2}(x)$ contains $\delta(x)$ distinct elements. The following lemma, which gives a necessary and sufficient condition for $0$ to belong to $ H^{c_2}(x)$, will be useful later on.

\begin{lemma}\label{when0}
Let
$(c_1,c_2)\in \{0,\ldots,l-1\}^2$ and let $x \in X^{c_1}$. 
We have $0 \in H^{c_2}(x)$  if and only if the following two conditions are satisfied:
\begin{enumerate}[(i)]
\item $x \notin X^{c_2}$; \smallbreak 
\item $\sharp\{ y\in X^{c_1} \ |\  y<x\} <
\sharp \{ y\in X^{c_2} \ |\  y<x\}$.
\end{enumerate}
\end{lemma}

\begin{proof}
If $x \in X^{c_2}$, then $0 \notin H^{c_2}(x)$. If $x \notin X^{c_2}$ and
$\sharp\{ y\in X^{c_1} \ |\  y<x\} \geq 
\sharp \{ y\in X^{c_2} \ |\  y<x\},$ 
then
$$ \sharp\{ y\in \Z\setminus X^{c_1} \ |\ 1-m< y<x\} \leq 
\sharp \{ y\in \Z\setminus X^{c_2} \ |\ 1-m<y<x\}.$$
We have $\delta(x)= \sharp\{ y\in \Z\setminus X^{c_1} \ |\  1-m<y<x\}$, whence $y_d^{c_2} <x$ for all $d=1,\ldots,\delta(x)$. We conclude that  
$0 \notin H^{c_2}(x)$. 

Suppose now that Conditions (i) and (ii) are satisfied. Then
$$ \sharp\{ y\in \Z\setminus X^{c_1} \ |\ 1-m< y<x\} >
\sharp \{ y\in \Z\setminus X^{c_2} \ |\ 1-m<y<x\}.$$
Since $\delta(x)= \sharp\{ y\in \Z\setminus X^{c_1} \ |\  1-m<y<x\}$ and $x \notin X^{c_2}$, there exists $d \in \{1,\ldots,\delta(x)\}$ such that $y_d^{c_2} =x$. Hence, $0 \in H^{c_2}(x)$. 
\end{proof}

Let $c \in\{0,1,\ldots,l-1\}$ and $x\in X^c$. Then, for $k\in \mathbb{Z}_{\geq 0}$, we set:
$$\mathcal{N}_k (x):=\left\{\begin{array}{ll}
    \sharp\{ t \in \{c+1,\ldots,l-1\} \ |\ x\notin X^t \}& \text{ if } k=0;    \\
    \sharp\{ t \in \{0,\ldots,l-1\} \ |\ x-ke\notin X^t \cup \Z_{\leq -m} \}& \text{ if } k>0.
    \end{array}\right.$$
The main result of this section is Proposition \ref{nu}, which produces the number of charged  hook lengths in $H(\ulambda)$ that are  congruent to $0$ modulo $e$ in terms of the numbers $\mathcal{N}_k (x)$. In order to prove it, we will first compute the number of charged  hook lengths equal to $0$ --- we can refer to this as the  case $e=\infty$.

\begin{Prop}\label{nuinfty}
The number of charged hook lengths in $H(\ulambda)$ equal to $0$ is
$$\sum_{0\leq c\leq l-1} \sum_{x \in X^c}   \mathcal{N}_0 (x).$$
\end{Prop}

\begin{proof}
Let $c \in \{0,\ldots,l-1\}$.
First, we observe that if $x \in {X}^c$, then $0 \notin H^c(x)$. 

Now let
$(c_1,c_2)\in \{0,\ldots,l-1\}^2$ such that $c_1<c_2$
and let $\mathcal{N}_0(X^{c_1},X^{c_2})$
be the  number of elements equal to $0$ inside
$$H(X^{c_1},X^{c_2}):=\left(\bigcup_{x \in {X}^{c_1}}
 H^{c_2}(x)\right) \bigcup \left(\bigcup_{x \in {X}^{c_2}}
 H^{c_1}(x)\right). 
$$
If $x \in {X}^{c_1} \cap {X}^{c_2}$, then 
$0 \notin H^{c_i}(x)$  for $i=1,2$. Moreover, since $s_{c_1} \leq s_{c_2}$, the set $X^{c_2}$ has $s_{c_2}-s_{c_1}$ more elements than $X^{c_1}$.

Set $M:=s_{c_2}-s_{c_1}$.
Let 
$x_1>x_2>\dots>x_N$ be the elements of $X^{c_1}\setminus {X}^{c_2} $ and
$z_1>z_2>\dots>z_{M+N}$ be the elements of  $X^{c_2}\setminus{X}^{c_1}$. 
Lemma \ref{when0} implies the following:
\begin{itemize}
\item For all $i=1,\ldots,M$, we have $0 \notin H^{c_1}(z_i)$. \smallbreak
\item For all $i=M+1,\ldots,M+N$, we have \smallbreak
        \begin{itemize}
        \item $0 \notin H^{c_2}(x_{i-M})$ and $0 \in H^{c_1}(z_i)$, if $x_{i-M} < z_i$;\smallbreak
        \item $0 \in H^{c_2}(x_{i-M})$ and $0 \notin H^{c_1}(z_i)$, if $x_{i-M} >z_i$.
         \end{itemize}
\end{itemize}
Therefore, $\mathcal{N}_0(X^{c_1},X^{c_2})=N=\sharp \{  x\in X^{c_1} \ |\ x\notin X^{c_2}\}$.

The result of the proposition follows from \eqref{general}
and the fact that
$$\sum_{0 \leq c_1 < c_2 \leq l-1} \mathcal{N}_0(X^{c_1},X^{c_2})= 
  \sum_{0\leq c\leq l-1} \sum_{x \in X^c}   \mathcal{N}_0 (x).$$
\end{proof}

\begin{Prop}\label{nu}
The number of charged hook lengths in $H(\ulambda)$ congruent  to $0$ modulo $e$ is 
$$\sum_{0\leq c\leq l-1} \sum_{x \in X^c}  \sum_{k\geq 0} \mathcal{N}_k (x).$$
\end{Prop}

\begin{proof}
Let $k \in \Z_{>0}$. For all $c \in \{0,\ldots,l-1\}$, we have
$X^c=(\beta^c_1,\beta^c_2,\ldots,\beta^c_{m+s_c})$, and we will write
$X^c[ke]$ for the tuple $(\beta^c_1+ke,\beta^c_2+ke,\ldots,\beta^c_{m+s_c}+ke,
-m+ke,-m+ke-1,-m+ke-2,\ldots,-m+1)$.

Let
$(c_1,c_2)\in \{0,\ldots,l-1\}^2$ such that $c_1<c_2$.
Note that, since ${\bf s} \in \mathcal{A}_e^l$, we have
$$s_{c_1} \leq s_{c_2}+ke \quad\text{and}\quad  s_{c_2} \leq s_{c_1} +ke.$$
Using again here the notation of the proof of Proposition \ref{nuinfty}, we obtain that the number of elements equal to $\pm ke$ inside $H(X^{c_1},X^{c_2})$ is
$$\mathcal{N}_0(X^{c_1},X^{c_2}[ke]) +\mathcal{N}_0(X^{c_2},X^{c_1}[ke]).$$
 
In a similar way, we obtain that the number of elements equal to $\pm ke$ inside $\bigcup_{x \in {X}^{c}}
 H^{c}(x)$ is $\mathcal{N}_0(X^{c},X^{c}[ke])$ for all $c \in \{0,\ldots,l-1\}$. 
Following \eqref{general},  we deduce that the number of charged hook lengths in $H(\ulambda)$ equal to $\pm ke$ is
$$\sum_{c_1=0}^{l-1} \sum_{c_2=0}^{l-1} \mathcal{N}_0(X^{c_1},X^{c_{2}}[ke]).$$

We now observe that, for all $(c_1,c_2)\in \{0,\ldots,l-1\}^2$, we have
$$\mathcal{N}_0(X^{c_1},X^{c_{2}}[ke])=\sharp \{x \in X^{c_1}\ | \ x \notin  
X^{c_{2}}[ke]\} = \sharp \{x \in X^{c_1}\ | \ x-ke \notin  
X^{c_{2}} \cup \Z_{\leq -m}\}.$$
Therefore, the 
number of charged hook lengths in $H(\ulambda)$ equal to $\pm ke$ is $$\sum_{0\leq c\leq l-1} \sum_{x \in X^c}   \mathcal{N}_k (x).$$ 
We conclude that the 
number of charged hook lengths in $H(\ulambda)$ congruent to $0$ modulo $e$ is
$$\sum_{0\leq c\leq l-1} \sum_{x \in X^c}  \sum_{k\geq 0} \mathcal{N}_k (x).$$
 \end{proof}

\begin{exa}
Let us look at the Example \ref{firstabacus}, whose $H(\ulambda)$ has been calculated in Example \ref{firstabacusprime}.
We have   
$X[\ulambda,3]=((2,-1,-2),(2,0,-1,-2),(3,2,0,-1,-2))$.
Since $X^0 \subset X^1 \subset X^2$, we have $\mathcal{N}_0(x)=0$ for all $x \in X^c$ with $c \in \{0,1,2\}$, and so $0 \notin H(\ulambda)$.

Take $e=2$. For $k>1$, we have $\mathcal{N}_k(x)=0$ for all $x \in X^c$ with $c \in \{0,1,2\}$.
For $k=1$, we have $\mathcal{N}_k(x)=0$ for all $x \in X^c$ with $c \in \{0,1,2\}$ except for $x \in \{\beta_1^0,\beta_1^1, \beta_1^2,\beta_2^2\}$. We have 
$\mathcal{N}_1(\beta_1^0)=1$, $\mathcal{N}_1(\beta_1^1)=1$, $\mathcal{N}_1(\beta_1^2)=3$ and $\mathcal{N}_1(\beta_2^2)=1$.
We deduce that the total number of even charged hook lengths in $H(\ulambda)$ is equal to $6$.

Take $e=3$. For $k>1$, we have $\mathcal{N}_k(x)=0$ for all $x \in X^c$ with $c \in \{0,1,2\}$.
For $k=1$, we have $\mathcal{N}_k(x)=0$ for all $x \in X^c$ with $c \in \{0,1,2\}$ except for $x = \beta_1^2$. We have 
$\mathcal{N}_1(\beta_1^2)=1$, and so the number of charged hook lengths in $H(\ulambda)$  congruent to $0$ modulo $3$ is equal to $1$.

For $e>3$, we have no charged hook lengths congruent to $0$ modulo $e$.
\end{exa}

\begin{exa}
Let us look at the Example \ref{second abacus}. We have   
$X[\ulambda,3]=((3,0,-2),(4,2,1,-1,-2))$. 
By Proposition \ref{nuinfty}, the number of charged hook lengths equal to $0$ is
$$\sum_{c=0}^1 \sum_{x \in X^c}   \mathcal{N}_0 (x) = \sharp \{x \in X^0\,|\,x \notin X^1\}=2.$$

Take $e=2$. Let $k \in \Z>0$ and let $x \in X^0 \cup X^1$. We have $\mathcal{N}_k(x)=0$ except in the following cases where we have $\mathcal{N}_k(x)=1$
\begin{itemize}
\item $k=1$ and $x \in \{\beta_1^0,\beta_1^1,\beta_2^1,\beta_3^1\}$;\smallbreak
\item $k=2$ and $x \in \{\beta_1^0,\beta_1^1\}$.
\end{itemize}
By Proposition \ref{nu}, the number of even charged hook lengths in $H(\ulambda)$ is equal to $8$.
\end{exa}

\subsection{Defect and weight  for Ariki-Koike algebras}\label{sec-def-wei}
Let $\ulambda\in \Pi^l$, ${\bf s}=(s_0,\ldots,s_{l-1})\in \mathbb{Z}^l$ and $e \in  \mathbb{Z}_{>1}$. 
From now on, we will refer to the number of elements of $H(\ulambda)$ congruent to $0$ modulo $e$ as the $e$-{\em defect} of $\ulambda$ attached to ${\bf s}$.
For $l=1$, this number is known  as the  $e$-{\em weight} of a partition.

In \cite{F}, Fayers introduced a notion of weight for a charged multipartition. 
We have now at least three 
 different ways to compute the weight $p_{(e,{\bf s})} (\ulambda)$ of $\ulambda$ attached to $(e,{\bf s})$. In this subsection, we will show that 
$p_{(e,{\bf s})} (\ulambda)$  is equal to the $e$-defect of $\ulambda$.
One consequence of this equality will be the proof of our main result in the case of cyclotomic Ariki--Koike algebras.

 Here are the three equivalent definitions of weight:
 \begin{itemize}
 \item { \bf The original definition \cite{F}:}
  $$p_{(e,{\bf s})} (\ulambda)=\sum_{0\leq i\leq l-1} c^{e,{\bf s}}_{s_i} (\ulambda) -\frac{1}{2} \sum_{i\in \mathbb{Z}/e\mathbb{Z}} (c^{e,{\bf s}}_i (\ulambda)-c^{e,{\bf s}}_{i-1} (\ulambda))^2$$
 where $c^{e,{\bf s}}_{i} (\ulambda)$ denotes the number of nodes in $\ulambda$ with content congruent to $i$ modulo $e$.  \smallbreak
 \item {\bf The quantum group definition \cite{F}:}  $$p_{(e,{\bf s})} (\ulambda) =  \displaystyle    \| \uemptyset \|^{(e,{\bf s})}  -\| \ulambda \|^{(e,{\bf s})} $$
 where $\|.\|$ is given in \cite[\S 3.1]{JL} (we will not need the precise definition here).
 \smallbreak
 \item {\bf The  abaci definition \cite{JL}:}
 $$p_{(e,{\bf s})} (\ulambda)=\omega_e ( \tau_{e,{\bf s}} (\ulambda))$$
 where $\omega_e$ is the usual $e$-weight of a partition
 and $\tau_{e,{\bf s}}:\Pi^l \to \Pi$ is {\em Uglov's map}, which is 
 easy to define  using  abaci (see \cite{JL}).
 \end{itemize}

The definition we will use in this paper is the last one. According to it,  the weight $p_{(e,{\bf s})} (\ulambda)$ can be defined recursively as follows:  Consider the 
 set of $\beta$-numbers $X:=X[\ulambda,m]=(X^{0},\ldots,X^{l-1})$
  of $\ulambda$ and set $w(X):=0$
 \begin{itemize}
 \item[Step 1:] 
 If 
 $X^{c-1} \subseteq X^{c}$ for all $c=1,\ldots,l-1$, 
 then we move to Step 3.
Otherwise, if there exists $x\in X^{c-1}$ such that $x \notin X^{c}$ for some $c \in \{1,\ldots,l-1\}$, then we define a set of $\beta$-numbers $Y=(Y^0,\ldots,Y^{l-1})$ such that
 $Y^{c-1} := X^{c-1} \setminus \{ x\}$, $Y^{c}:=X^{c}\sqcup \{ x\}$ and
 $Y^j:=X^j$ for all $j\neq c-1,c$. We set $w(Y):=w(X)+1$ and we move to Step 2. \smallbreak
 
 \item[Step 2:] We set $X:=Y$ and $w(X):=w(Y)$, and we go back to Step 1.
 \smallbreak

  \item[Step 3:]  If $\{x-e\,|\, x\in X^{l-1} \cap \Z_{>e-m}\} \subseteq X^0$, then we move to Step 5. Otherwise, if there exists $x \in X^{l-1}$ such that   $x-e \notin X^0 \cup \Z_{\leq -m}$,
 then we define a set of $\beta$-numbers $Y=(Y^0,\ldots,Y^{l-1})$ such that
 $Y^{l-1} := X^{l-1} \setminus \{ x\}$, $Y^{0}:=X^{0}\sqcup \{ x-e\}$ and
 $Y^j:=X^j$ for all $j\neq 0,l-1$. We set $w(Y):=w(X)+1$ and we move to Step 4. \smallbreak
 
 \item[Step 4:]  We set $X:=Y$ and $w(X):=w(Y)$, and we go back to Step $1$. \smallbreak
 
 \item[Step 5:] We set $p_{(e,{\bf s})} (\ulambda):=w(X)$.

 \end{itemize} 

It is now straightforward to see that the weight  is given by the same formula as  in Proposition \ref{nu}, and so we have the following:

\begin{Th}\label{weight-defect}
Let $\ulambda\in \Pi^l$, ${\bf s} \in \mathbb{Z}^l$ and $e \in  \mathbb{Z}_{>1}$.  We have that the $e$-weight  $p_{(e,{\bf s})} (\ulambda)$ of $\ulambda$ is equal to the $e$-defect of $\ulambda$ attached to ${\bf s}$, that is, the number of charged hook lengths in $H(\ulambda)$ congruent to $0$ modulo $e$.
\end{Th}

Now, by 
  \cite{F} (see also \cite{JL} for another proof), any two $l$-partitions that are in the same block have the same $e$-weight, and thus the same $e$-defect. We have hence  proved \eqref{equivmain}, which in turn implies the validity of Theorem \ref{main-AK}.

\subsection{$e$-cores and Schur elements} 
  Let $e \in  \mathbb{Z}_{>1}$. Another fundamental notion associated to the blocks  of Ariki-Koike algebras is the notion of $e$-core. Since Theorem \ref{weight-defect} allows us to define the notion of $e$-weight  using the Schur elements, it is natural to ask whether the same thing can be done with the $e$-core.

 The $e${\em-core} of a partition $\lambda \in \Pi$ is the partition obtained after removing all the $e$-hooks, that is, all hooks of length $e$. 
  Two simple modules $V^{\lambda}$ and $V^{\mu}$ (with $\lambda$ and $\mu$ of the same rank) are in the same 
   block if and only if they have the same $e$-core.

\begin{Prop}\label{BGO}
   If $\lambda \in \Pi$and $\mu$ is the $e$-core of $\lambda$, then
   the Schur element $s_\mu(q)$ of $V^{\mu}$ divides the Schur element $s_\lambda(q)$ of $V^{\lambda}$ in $\Z[q,q^{-1}]$. 
   \end{Prop}
\begin{proof}
The proof is based on a result by Bessenrodt, Gramain and Olsson. 
  By \cite[Theorem 4.4]{BGO},
  the multiset  $H (\mu)$ of hook lengths
   of $\mu$ is contained in the multiset $H(\lambda)$ of hook lengths of $\lambda$. The result is a straightforward consequence of the description of Schur elements given by Formula \eqref{claim}.  
  \end{proof}
  
  \begin{Rem}
  If $\lambda \in \Pi(n)$ and $q=1$, then we are in the group algebra case and
  $s_\lambda(1)=n!/\chi_{V^\lambda}(1)$. Using the famous hook length formula for $\chi_{V^\lambda}(1)$ yields that $s_\lambda(1)$ is equal to the product of the hook lengths of $\lambda$. Therefore, if $\mu$ is the $e$-core of $\lambda$, then
  $s_{\lambda}(1)/s_{\mu}(1)$ is equal to the product of the elements of 
  $H(\lambda) \setminus H (\mu)$. This is stated in \cite[Corollary 4.12]{BGO}, where it is also observed that the relative hook formula obtained in  \cite[Theorem 9.1]{MN} is a particular case of this result.
  \end{Rem}

   When $l\in \mathbb{Z}_{>1}$, there is a similar (but a little bit more complicated) notion of $e$-core for $\ulambda \in \Pi^l$, which has been given in \cite{JL}
   and yields the same properties as far as blocks are concerned. 
In fact, if ${\bf s} \in \mathbb{Z}^l$, we say that $\ulambda \in \Pi^l$ is an $(e,{\bf s})${\em-core} if $p_{(e,{\bf s})} (\ulambda)=0$. 
Following Theorem \ref{weight-defect}, the $(e,{\bf s})$-cores are exactly the $l$-partitions with no charged hook lengths divisible by $e$  (this is a well-known result for $l=1$). 
Therefore, Theorem \ref{weight-defect}  allows us to determine whether an $l$-partition is an $(e,{\bf s})$-core by looking at its Schur element:

\begin{Prop}
Let $\ulambda \in \Pi^l(n)$  and ${\bf s}=(s_0,\ldots,s_{l-1}) \in \mathbb{Z}^l$.
We have that  $\ulambda $ is an $(e,{\bf s})$-core if and only if $\Phi_e (y)$ does not  divide $s_{\ulambda} ({\bf y})$,
 the Schur element associated with the simple module $V^{\ulambda}$  of the Ariki--Koike algebra $\mathcal{H}_n^{\bf y}$ where
 ${\bf y}=(y^{s_{0}},\ldots ,y^{s_{l-1}}\;;\, y)$. 
\end{Prop}

\begin{proof}
By Theorem \ref{weight-defect},  $p_{(e,{\bf s})} (\ulambda)$ is equal to the $e$-defect of $\ulambda$ attached to ${\bf s}$.
The latter is equal to the $e$-defect of $V^{\ulambda}$ (that is, the maximal $N \in \N$ such that $\Phi_e(y)^N$ divides $s_{\ulambda} ({\bf y})$) as long as
$s_{\ulambda} ({\bf y}) \neq 0$.

If $\ulambda $ is an $(e,{\bf s})$-core, then $0 \notin H(\ulambda)$ and so 
 $s_{\ulambda} ({\bf y}) \neq 0$. Thus, the $e$-defect of $V^{\ulambda}$ is equal to $0$, whence $\Phi_e (y)$ does not  divide $s_{\ulambda} ({\bf y})$.
 
  Conversely, if $\Phi_e (y)$ does not  divide $s_{\ulambda} ({\bf y})$, then
 $s_{\ulambda} ({\bf y}) \neq 0$. Thus, the $e$-defect of ${\ulambda}$ attached to ${\bf s}$ is  equal to $0$, whence $p_{(e,{\bf s})} (\ulambda)=0$.
 \end{proof}


%

\section{Defect in cyclotomic Hecke algebras of $G(l,p,n)$}

We can now deduce our main result for type $G(l,p,n)$ from  $G(l,1,n)$ with the use of Clifford theory. For more details about Clifford theory in general the reader may refer to \cite[\S2.3]{C}, and in this particular setting to \cite{GJa}, \cite{CJ} or \cite{Ma1}.

\subsection{Generic Hecke algebras of $G(l,p,n)$}
Let $l,p,n$ be three positive integers such that $d:=l/p \in \Z$.  
By definition, $G(l,p,n)$  is the group of all 
  $n \times n$ monomial matrices whose non-zero entries are ${l}$-th roots of unity, while the product of all non-zero
  entries is a $d$-th root of unity. Therefore, for $n=1$, we have $G(l,p,1) \cong G(d,1,1) \cong \Z/d\Z$, and this case has been covered in the previous section.  Moreover,  for $n=2$, the case where $p$ is even cannot be treated with the use of Clifford theory, and it is an unfortunate but  usual exception to this kind of results. This is why,
 from now on, we assume that
\begin{itemize}
\item  either $n>2$,\smallbreak
\item or $n=2$  and  $p$ is odd.
\end{itemize}

Let ${\bf x}:=(X_0,\,\ldots,\,X_{d-1}\,;\,x)$ be a set of $d+1$ indeterminates  and
set $R:=\mathbb{Z}[{\bf x},{\bf x}^{-1}]$.
The {\it generic Hecke algebra} $\mathcal{H}_{l,p,n}$ of $G(l,p,n)$ is the $R$-algebra with generators $t_0,\,t_1,\,\ldots,\,t_{n}$ and relations:
\begin{itemize}
\item $
(t_0-X_0)(t_0-X_1)\cdots (t_0-X_{d-1})=0$, \smallbreak
\item $(t_j-x)(t_j+1)=0$ for  $j=1,...,n$, \smallbreak
\item $t_1t_3t_1=t_3t_1t_3$, $t_jt_{j+1}t_j=t_{j+1}t_jt_{j+1}$ for $j=2,\ldots,n-1$, \smallbreak
\item $t_1t_2t_3t_1t_2t_3=t_3t_1t_2t_3t_1t_2$, \smallbreak
\item $t_1t_j=t_jt_1$ for $j=4,\ldots,n$, \smallbreak
\item $t_i t_j=t_j t_i$  for $2 \leq i <j \leq n$ with $j-i>1$, \smallbreak
\item $t_0t_j=t_jt_0$ for $j=3,\ldots,n$, \smallbreak
\item $t_0t_1t_2=t_1t_2t_0$, \smallbreak
\item $\underbrace{t_2t_0t_1t_2t_1t_2t_1\ldots}_{p+1 \textrm{ factors}}=
\underbrace{t_0t_1t_2t_1t_2t_1t_2\ldots}_{p+1 \textrm{ factors}}$\,. \smallbreak
\end{itemize}

\subsection{Cyclotomic Hecke algebras of $G(l,p,n)$ and Clifford theory}
Let  $\eta_{l}:={\rm exp}(2\pi i/l)$, $\eta_d:=\eta_l^p$ and $K:=\mathbb{Q}(\eta_l)$. 
Let $y$ be an indeterminate, let $m \in \Z$ and
and let $\varphi_m: \mathbb{Z}_K[{\bf x},{\bf x}^{-1}] \rightarrow \mathbb{Z}_K[y,y^{-1}]$   be a cyclotomic specialisation of $\mathcal{H}_{l,p,n}$ such that
\begin{eqnarray*}\label{cy2}
&&\varphi_m(X_j)=\eta_{d}^{j}y^{mr_{j}} \textrm{ for}\ j=0,1,\ldots ,d-1, \\
&&\varphi_m(x)=y^{mr} 
\end{eqnarray*}
where $(r_{0},\ldots ,r_{d-1},r)\in \mathbb{Z}^{d+1}$. The cyclotomic Hecke algebra $K(y) (\mathcal{H}_{l,p,n})_{\varphi_m}$ is split semisimple. Note that the representation theory of $(\mathcal{H}_{l,p,n})_{\varphi_m}$, and especially its block theory,  does not depend much on $m$ (see also \cite[\S 5.5.1]{C}).

The algebra  $(\mathcal{H}_{l,p,n})_{\varphi_p}$ can be viewed as a subalgebra  of the cyclotomic Ariki--Koike algebra
$(\mathcal{H}_{n}^{{\bf q}})_\varphi$ of type $G(l,1,n)$ where $\varphi: \mathbb{Z}_K[{\bf q},{\bf q}^{-1}] \rightarrow \mathbb{Z}_K[y,y^{-1}]$ is the 
cyclotomic specialisation given by
\begin{eqnarray*}\label{cy3}
&&\varphi(Q_i)=\eta_{l}^{i}y^{r_{i}} \textrm{ for}\ i=0,1,\ldots ,l-1, \\
&&\varphi(q)=y^{pr} 
\end{eqnarray*}
with $r_{kd+j}:=r_j$ for all $j=0,\ldots,d-1$ and $k=1,\ldots,p-1$.
In particular,  if we denote by $G$ the cyclic group of order $p$, $(\mathcal{H}_{n}^{{\bf q}})_\varphi$ is a
``twisted symmetric algebra'' of $G$ over $(\mathcal{H}_{l,p,n})_{\varphi_p}$  (see again \cite[\S 5.5.1]{C}). This implies an action  of $G$ on $\mathrm{Irr}(K(y)(\mathcal{H}_{n}^{{\bf q}})_\varphi)$ which corresponds to the action generated by the cyclic permutation by $d$-packages on the $l$-partitions of $n$:
$$\begin{array}{rl}
\sigma: &\,\,\,\ulambda=(\lambda^{0},\ldots,\lambda^{d-1},\lambda^{d},\ldots,\lambda^{2d-1},\ldots,\lambda^{pd-d},\ldots,\lambda^{pd-1})\\ &\\
\mapsto &{\sigma}(\ulambda)=(\lambda^{pd-d},\ldots,\lambda^{pd-1},\lambda^{0},\ldots,\lambda^{d-1},\ldots, \lambda^{pd-2d},\ldots,\lambda^{pd-d-1}).\smallbreak
\end{array}$$

From now on, to make the notation lighter, we will write $\mathcal{H}$ for $K(y)(\mathcal{H}_{n}^{{\bf q}})_\varphi$ and $\bar{\mathcal{H}}$ for 
$K(y) (\mathcal{H}_{l,p,n})_{\varphi_p}$. 
Let $M \in \mathrm{Irr}(\bar{\mathcal{H}})$. By Clifford theory, there exists $V^{\ulambda} \in
 \mathrm{Irr}({\mathcal{H}})$ such that $M$ is a composition factor of $\mathrm{Res}_{\bar{\mathcal{H}}}^{\mathcal{H}}(V^{\ulambda})$.  Moreover, there is an action of $G$ on $\mathrm{Irr}(\bar{\mathcal{H}})$ such that, if we denote by $\bar{\Omega}_{\ulambda}$ the orbit of $M$ under the action of $G$, we have
 $$[\mathrm{Res}_{\bar{\mathcal{H}}}^{\mathcal{H}}(V^{\ulambda})]=\sum_{E \in \bar{\Omega}_{\ulambda}}[E]$$
 in the Grothendieck group of the category of finite dimensional $\bar{\mathcal{H}}$-modules.
 
 Let $V \in  \mathrm{Irr}({\mathcal{H}})$.
 The elements of $\bar{\Omega}_{\ulambda}$ appear as composition factors in
 $\mathrm{Res}_{\bar{\mathcal{H}}}^{\mathcal{H}}(V)$ if and only if $V=V^{g(\ulambda)}$ for some
 $g \in G$.
 In particular, we have
 \begin{equation}\label{howow}
 [\mathrm{Res}_{\bar{\mathcal{H}}}^{\mathcal{H}}( V^{\sigma(\ulambda)})]=[\mathrm{Res}_{\bar{\mathcal{H}}}^{\mathcal{H}}(V^{\ulambda})].
 \end{equation}
We deduce that
$$\mathrm{Irr}(\bar{\mathcal{H}})=\{ E\,|\,E \in  \bar{\Omega}_{\ulambda}, \, \ulambda \in \Pi^l(n)\}.$$

\subsection{Blocks and defect} 
Let $\theta: \mathbb{Z}_{K}[y,y^{-1}] \rightarrow K(\eta), y \mapsto \eta$ be a specialisation of $\mathbb{Z}_{K}[y,y^{-1}]$ such that  $\eta$ is a non-zero complex number and $\theta(\eta_l)$ is a primitive $l$-th root of unity. 
The $\theta$-blocks of $\bar{\mathcal{H}}$ have been determined by Wada  \cite{W}, and they can be described as follows:
Let $E,F \in \mathrm{Irr}(\bar{\mathcal{H}})$ with $E \in  \bar{\Omega}_{\ulambda}$ and $F \in  \bar{\Omega}_{\umu}$ for some $\ulambda,\umu \in \Pi^l(n)$.
\begin{itemize}
\item  If $\{V^\ulambda\}$ is a $\theta$-block, then $\{E\}$ is a $\theta$-block. \smallbreak
\item Otherwise, $E$ and $F$   are in the same $\theta$-block if and only if 
 $V^\ulambda$ and $V^{g(\umu)}$ are in the same $\theta$-block for some $g\in G$.
\end{itemize}

On the other hand, thanks to Clifford theory, the Schur elements of $\bar{\mathcal{H}}$ can be easily obtained from the Schur elements of $\mathcal{H}$ as follows (see
 \cite[Proposition 2.3.15]{C}):
  \begin{equation}\label{schurmultiple}
s_E(y)=\frac{| \bar{\Omega}_{\ulambda} |}{p} \varphi(s_\ulambda({\bf q})).
\end{equation}
The above equation, in combination with \eqref{howow}, implies that
 \begin{equation}\label{howcanthisow}
\varphi(s_\ulambda({\bf q})) = \varphi(s_{\sigma(\ulambda)}({\bf q})).
\end{equation}
In fact, the last equality can be obtained independently of Clifford theory by  \cite[Proposition 2.5]{CJG}

We can now prove Conjecture \ref{our conj} for the groups $G(l,p,n)$:

\begin{Th}\label{main-AK2}
Let $\Phi$ denote the minimal polynomial of $\eta$ over the field $K$
and let  $E,F \in \mathrm{Irr}(\bar{\mathcal{H}})$. If $E$ and $F$ are in the same $\theta$-block, then they have the same $\Phi$-defect, that is,
$\nu_\Phi(s_E(y))=\nu_\Phi(s_F(y))$.
\end{Th}

\begin{proof} 
If $E$ and $F$ are in the same $\theta$-block, then, as we mentioned just above, there exists $g \in G$ such that 
 $V^\ulambda$ and $V^{g(\umu)}$ are in the same $\theta$-block.
 By Theorem \ref{main-AK}, we have that $\nu_\Phi(\varphi(s_{\ulambda}({\bf q})))=\nu_\Phi(\varphi(s_{g(\umu)}({\bf q})))$. The result follows from
 \eqref{schurmultiple} and \eqref{howcanthisow}.
   \end{proof}

\section{Defect in cyclotomic Hecke algebras of exceptional groups}

As far as the exceptional complex reflection groups are concerned, there are no general results about decomposition matrices and blocks as in the case of $G(l,p,n)$.
For the exceptional real reflection groups, in the equal parameter case, the blocks are given in \cite[Appendix F]{GP}. Even though the conceptual proof by Geck \cite{G} of Conjecture \ref{our conj} for Weyl groups precedes the explicit description of blocks for the exceptional real reflection groups, one can also verify the conjecture for these groups by inspecting the corresponding blocks. The only other work on blocks of cyclotomic Hecke algebras associated with exceptional complex reflection groups is 
\cite{CM}. In that paper, the cyclotomic Hecke algebras considered are the ones
arising from generalised Harish-Chandra theory, for which the term ``cyclotomic Hecke algebras'' was originally conceived. The choices of parameters for these cyclotomic Hecke algebras are given by Brou\'e and Malle in \cite{BM, Ma2} and 
for the exceptional complex reflection groups of rank $2$ are as follows:

\begin{center}
$ $\\
\begin{tabular}{|c|c|c|}
  \hline
  $W$ & $K$ &  $\textrm{Parameters of } \mathcal{H}_\varphi(W)$ \\ 
  \hline
  $G_{4}$ & $\Q(\eta_3)$ &$1,y,y^2$ \\
  \hline
   $G_{5}$& $\Q(\eta_3)$  &  $1,y,y^2\,;\,1,y,y^2$ \\
   & &  $1,y,y^2\,;\,1,y^2,y^4$ \\
   &&  $1,y,y^2\,;\,1,y^4,y^8$ \\
  \hline
    $G_{8}$ &   $\Q(i)$ & $1,\eta_8^3y,\eta_8^5y,y^2$ \\
  &  &  $1,y,-y,y^2$ \\
   &&  $1,y,y^2,y^3$ \\
    &&   $1,y,-y,-y^4$ \\
    &&  $1,y^3,-y^3,-y^4$ \\
    &&   $1,-1,-y,y^5$ \\
    &&            $1,-y^4,y^5,-y^5$ \\
    &&              $1,-y,-y^4,y^5$ \\
\hline
    $G_{9}$ & $\Q(\eta_8)$  & $1,y^4\,;\,1,y^2,y^4,y^6$ \\
 \hline
    $G_{10}$ & $\Q(\eta_{12})$   & $1,-y^2,y^4\,;\,1,y^3,-y^3,y^6$ \\
  \hline
    $G_{12}$ & $\Q(i\sqrt{2})$    & $1,y^2$ \\
    \hline
    $G_{16}$ & $\Q(\eta_5)$ & $1,y,y^2,y^3,y^4$ \\
    \hline
    $G_{20}$ & $\Q(\eta_3,\sqrt{5})$  & $1,y,y^2$ \\
    \hline
    $G_{22}$  & $\Q(i,\sqrt{5})$ & $1,y^2$ \\         
\hline   
\end{tabular}
$ $\\$ $\\
TABLE 1. \emph{The ``Zyklotomische Heckealgebren'' of the rank $2$ exceptional groups}\\
\end{center}\medskip

In \cite{CM}, the authors calculated the decomposition matrices for all the algebras above and for all specialisations $\theta: y \mapsto \eta$ that yield non-semisimple algebras in characteristic $0$. They also showed that all modular irreducible representations can be lifted to ordinary ones, that is, the decomposition matrix is of the form:
$$D_\theta=\left(
\begin{array}{cccc}
1 & 0 & \dots &0\\
0& 1 & \dots &0\\
\vdots & \vdots & \ddots &\vdots\\
0 & 0 & \dots &1\\
* & * & \dots &*\\
\vdots & \vdots & \ddots &\vdots\\
* & * & \dots &*\\
\end{array}\right)$$
$ $\\
The subset  $\mathcal{B}_\theta$ of $\Irr(W)$ labelling the rows of the identity matrix at the top of $D_\theta$ is called an \emph{optimal basic set}. More formally, $\mathcal{B}_\theta$ is a subset of $\Irr(W)$ such that
\begin{enumerate}
\item there exists a bijection $\mathcal{B}_\theta \leftrightarrow \Irr(K(\eta)
\mathcal{H}_\varphi(W)),\,V \mapsto M_V$\,;\smallbreak
\item we have $[V:M_{V'}]=\delta_{V,V'}$ for all $V,V' \in \mathcal{B}_\theta$.
\end{enumerate}
To save space, the information given in \cite{CM} is the shape of each block of $D_\theta$, as well as the representations of the block that belong to $\mathcal{B}_\theta$. This is why, in order to verify Conjecture \ref{our conj} in all these cases, we created a computer function that first recovered all other representations inside the block, and then calculated their defect.

\subsection{The algorithm} 
The algorithm is very elementary and we chose to use GAP3 to implement it, because the package CHEVIE \cite{chevie1,chevie2} contains the irreducible representations and the Schur elements in factorised form for any Hecke algebra associated to a complex reflection group. However, any language with this information could do. The function, named ``DefBlock'', and its outputs can be found on the project's webpage \cite{Web}.


Let $\theta: \mathbb{Z}_{K}[y,y^{-1}] \rightarrow K(\eta), y \mapsto \eta$ be a $\Z_K$-algebra morphism such that the specialised algebra $K(\eta)\mathcal{H}_\varphi(W)$ is not semisimple. Then $\eta$ is a root of unity of order $e>0$ and there exists $r \in \{1,2,\ldots,e\}$ with ${\rm gcd}(e,r)=1$ such that $\eta=\eta_e^r$.  Let $\Phi$ denote the minimal polynomial of $\eta$ over $K$. Our function  takes as inputs the cyclotomic Hecke algebra $\mathcal{H}_\varphi(W)$ with parameter $y$, the integers $e$ and $r$, and the list $l$ of irreducible representations in a block $B$ that belong to the optimal basic set $\mathcal{B}_\theta$. We only need to consider blocks of non-zero defect. 

The first step is to determine all irreducible representations in the block $B$ (this is information that the first author should have stored somewhere, but did not). We calculate all possible linear combinations of the characters of the representations in the list $l$, with $y$ specialised to $\eta_e^r$. Since the dimensions of the irreducible representations of the exceptional complex reflection groups we consider are no higher than $6$, taking linear combinations of up to $6$ characters is enough. It turns out that this suffices to uniquely determine which irreducible representations belong to the block $B$.

The second step is straightforward, given that one has the Schur elements in factorised form. For each irreducible representation in the block $B$, we calculate  its $\Phi$-defect as the multiplicity of  $\eta_e^r$ as a root of its Schur element. We then determine the set of all $\Phi$-defect values inside $B$.

%

The outputs of the function are a list with the set of all $\Phi$-defect values and all representations inside the block $B$. Let us see an example:

\begin{exa}
We would like to calculate the defect of the block 
of $\mathcal{H}_\varphi(G_4)$ generated by $\chi_{1,0}$ and $\chi_{1,8}$ for $\eta=\eta_{12}$. We type
 \begin{verbatim}
gap> W:=ComplexReflectionGroup(4);;
gap> H:=Hecke(W,[[1,y,y^2]]);;
gap> DefBlock(H,12,1,[[1,0],[1,8]]);
[ [ 1 ], [ [ [ 1, 0 ] ], [ [ 1, 8 ] ], [ [ 2, 3 ] ] ] ]
\end{verbatim}
We know from \cite{CM} that the corresponding block of the decomposition matrix is of the form
$$\left(\begin{array}{cc}
1 & 0\\
0 & 1\\
1 & 1 
\end{array}\right)$$
where the first two rows are labelled by  $\chi_{1,0}$ and $\chi_{1,8}$. We now obtained that the last row is labelled by $\chi_{2,3}$ and that all representations in this block are of defect $1$.
\end{exa}

As we said in the beginning of this subsection, we ran this function for all blocks of non-zero defect of the cyclotomic Hecke algebras of Table 1, and stored the results on \cite{Web}.
If the cardinality of the set of all $\Phi$-defect values of a block is equal to $1$, then all irreducible representations in  the block have the same $\Phi$-defect. This, in combination with our results,  allowed us to deduce the following theorem:

\begin{Th} 
Let $W \in \{G_4,G_5,G_8,G_9,G_{10},G_{12},G_{16},G_{20},G_{22}\}$ and let $\mathcal{H}_\varphi(W)$ be one of the cyclotomic Hecke algebras of Table 1.
 If  $V,V' \in \Irr(W)$ belong to the same $\theta$-block, then they have the same $\Phi$-defect.
\end{Th}

\begin{Cor}
Let $W \in \{G_{12}, G_{22}\}$ and let $\mathcal{H}_\varphi(W)$ be any cyclotomic Hecke algebra associated with $W$. If  $V,V' \in \Irr(W)$ belong to the same $\theta$-block, then they have the same $\Phi$-defect.
\end{Cor}

\begin{proof}
If $W \in \{G_{12}, G_{22}\}$, then $W$ is generated by elements of order $2$ (that is, $W$ is a $2$-reflection group) that belong to a single conjugacy class. Therefore, the generic Hecke algebra $\mathcal{H}(W)$, defined over a suitable base ring, is isomorphic to the algebra $\mathcal{H}_\varphi(W)$ of Table 1, and every possible specialisation has already been considered by \cite{CM} and our program.
\end{proof}

Before we finish this section, we should mention here that there is another work on decomposition matrices of exceptional complex reflection groups, Chavli's paper \cite{Ch20} on the generic Hecke algebras of $G_4$, $G_8$ and $G_{16}$, where she proves the existence of optimal basic sets for all possible specialisations of these algebras. The decomposition matrices of all cyclotomic Hecke algebras associated with these three groups could theoretically be studied through Chavli's paper, but there are too many options to be considered if we want to verify Conjecture \ref{our conj} case-by-case. We will return to this point when we discuss future perspectives in \S \ref{perspectives}.

\section{Defect in other types of essential algebras}

In \S \ref{sec-essential} we discussed about essential algebras, which are symmetric algebras whose Schur elements are Laurent polynomials of a specific form. Hecke algebras are particular cases of essential algebras.  One may ask whether Conjecture \ref{our conj} holds for other types of essential algebras. The answer is positive for Yokonuma--Hecke algebras (of type $A$), which were introduced by Yokonuma \cite{yo} as generalisations of Iwahori--Hecke algebras.

\subsection{The Yokonuma--Hecke algebra of type $A$ and the cyclotomic Yokonuma-Hecke algebra} The Yokonuma--Hecke algebra of type $A$ is a generalisation of the Iwahori--Hecke algebra of type $A$ and can be seen as a particular case of a cyclotomic Yokonuma--Hecke algebra, introduced in \cite{CPA2}. The term ``cyclotomic'' here does not refer to a cyclotomic specialisation, but to the fact that the cyclotomic Yokonuma--Hecke algebra generalises the Yokonuma--Hecke algebra of type $A$ in the same way that the Ariki--Koike algebra, which is also called ``cyclotomic Hecke algebra'' by several people, generalises the Hecke algebra of type $A$. In order not to repeat everything twice, we will work directly with cyclotomic Yokonuma--Hecke algebras, but,  for all results used, we will also give the references for the Yokonuma--Hecke algebra of type $A$, because it was studied first.

 Let  $n\in \mathbb{Z}_{\geq 0}$ and $d,l \in \mathbb{Z}_{>0}$. 
Let ${\bf q}:=(Q_0,\,\ldots,\,Q_{l-1}\,;\,q)$ be a set of $l+1$ indeterminates  and
set $R:=\mathbb{Q}[{\bf q},{\bf q}^{-1}]$. We define the \emph{cyclotomic
Yokonuma--Hecke algebra} ${\rm Y}(d,l,n)$ as the associative  $R$-algebra (with unit) with a presentation by generators:
$$g_1,g_2,\ldots,g_{n-1}, t_1,\ldots, t_n, X_1$$
and relations:
$$\begin{array}{rclcl}
g_ig_j & = & g_jg_i && \mbox{for all $i,j=1,\ldots,n-1$ such that $\vert i-j\vert > 1$,}\\[0.1em]
g_ig_{i+1}g_i & = & g_{i+1}g_ig_{i+1} && \mbox{for  all $i=1,\ldots,n-2$,}\\[0.1em]
t_it_j & =  & t_jt_i &&  \mbox{for all $i,j=1,\ldots,n$,}\\[0.1em]
g_it_j & = & t_{s_i(j)}g_i && \mbox{for all $i=1,\ldots,n-1$ and $j=1,\ldots,n$,}\\[0.1em]
t_j^d   & =  &  1 && \mbox{for all $j=1,\ldots,n$,}\\[0.2em]
g_i^2  & = & q +(q-1)  e_{i}  g_i && \mbox{for  all $i=1,\ldots,n-1$},\\[0.1em]
X_1g_1X_1g_1&=&g_1X_1g_1X_1&&\\[0.1em]
X_1g_i&=&g_iX_1 &&  \mbox{for all $i=2,\ldots,n-1$,}\\[0.1em]
X_1t_j&=&t_jX_1 &&  \mbox{for all $j=1,\ldots,n$,}\\[0.1em]
(X_1 -Q_0) \cdots(X_1 -Q_{l-1})&=&0&&
\end{array}\\
$$
where, for all $i=1,\ldots,n-1$, $s_i$ is the transposition $(i,i+1)$ and 
$$e_i=\frac{1}{d}\sum\limits_{0\leq s\leq d-1}t_i^s t_{i+1}^{-s}\ .$$
For $d=1$, the generators $t_j$ disappear and ${\rm Y}(1,l,n)$ is isomorphic to the  Ariki--Koike algebra $R\,\mathcal{H}^{\bf q}_{n}$. For $l=1$, the generator $X_1$ disappears and ${\rm Y}(d,1,n)$ is the Yokonuma--Hecke algebra of type $A$, which in turn becomes the group algebra of $G(d,1,n)$ for $q=1$. The presentation of ${\rm Y}(d,1,n)$ by generators and relations is due to \cite{Ju1, JuKa, Ju2}, with the simplified quadratic relation for the $g_i$ being due to \cite{CPA} (see also \cite[Remark 3.1]{CP} for the quadratic relation given above).

The representation theory of the Yokonuma--Hecke algebra ${\rm Y}(d,1,n)$ of type $A$ was first studied in \cite{thi, thi2, thi3}, but its irreducible representations were explicitly constructed in \cite{CPA} and adapted in \cite{CP} for the quadratic relation given above. The irreducible representations of ${\rm Y}(d,l,n)$ were constructed in \cite{CPA2} and can be adapted similarly. 
It turns out that the algebra ${\rm Y}(d,l,n)$ is split semisimple over the field $K({\bf q})$, where $K:=\Q(\eta_{dl})$ (the splitting field is $\Q(\eta_{d},{\bf q})$). Moreover, there exists a bijection $\Pi^{dl} (n) \leftrightarrow \Irr(K({\bf q}){\rm Y}(d,l,n)),\, \ulambda  \mapsto V^{\ulambda}$.

A  symmetrising trace $\boldsymbol{\tau}$ was  defined in \cite{CPA2} on $K({\bf q}){\rm Y}(d,l,n)$; this map satisfies $\boldsymbol{\tau}(b)=\delta_{1b}$ for all $b$ in a certain basis of ${\rm Y}(d,l,n)$. The map  $\boldsymbol{\tau}$ coincides with the canonical symmetrising trace on $\mathcal{H}^{\bf q}_{n}$ for $d=1$ and with the symmetrising trace defined on the Yokonuma--Hecke algebra of type $A$ in \cite{CPA} for $l=1$. By \cite[Proposition 7.4]{CPA2}, the Schur element of  $V^{\ulambda} \in \Irr(K({\bf q}){\rm Y}(d,l,n))$ is equal to
\begin{equation}\label{YokonumaSchur}
d^n s_{\ulambda[1]}({\bf q}) s_{\ulambda[2]}({\bf q}) \ldots s_{\ulambda[d]}({\bf q})
\end{equation}
where, for all $i=1,\ldots,d$, $\ulambda[i]=(\lambda^{(i-1)l}, \lambda^{(i-1)l+1}\ldots,\lambda^{il-1}) \in \Pi^l(n)$ and $s_{\ulambda[i]}({\bf q})$ is given by Formula \eqref{claim}.  Our notation for the $\ulambda[i]$ is in agreement with the notation used in \S \ref{sec-U} if we take $\mathcal{U}=\left\{0,1,\ldots ,{dl-1}\right\}$, $t=d$ and $\mathcal{U}_i=\{ (i-1)l+k\,|\,k=0,1,\ldots,l-1\}$ for all $i=1,\ldots,d$.

This connection between the Schur elements of ${\rm Y}(d,l,n)$ and those of Ariki--Koike algebras was subsequently explained by the following isomorphism of $R$-algebras proved independently in \cite{PdA} and in \cite{Salim}:  
\begin{equation}\label{isom thm}
 {\rm Y}(d,l,n) \cong \bigoplus_{\underset{n_1,\ldots ,n_d\geq 0}{n_1+\ldots +n_d=n}} {\rm Mat}_{\frac{n!}{n_1!\dots n_d!}} \left(\mathcal{H}^{\bf q}_{n_1}\otimes_R   \mathcal{H}^{\bf q}_{n_2}\otimes_R  \ldots \otimes_R   \mathcal{H}_{n_d}^{\bf q}\right).
\end{equation} 
This isomorphism in the case of the Yokonuma--Hecke algebra of type $A$, that is, the case $l=1$, had  first been  obtained in \cite{Lus}, subsequently in \cite{JPA}, and adapted in \cite{CP} for the quadratic relation given above and the ring $R$. Another proof for $l=1$ was later given in \cite{EsRy}.
The existence of the isomorphism \eqref{isom thm} implies that the map $\boldsymbol{\tau}$ above is indeed a symmetrising trace on ${\rm Y}(d,l,n)$ over  $R$.

Let $y$ be an indeterminate and 
and let $\varphi: K[{\bf q},{\bf q}^{-1}] \rightarrow K[y,y^{-1}]$   be a $K$-algebra morphism such that
\begin{eqnarray*}\label{cy}
&&\varphi(Q_i)=\eta_{l}^{i}y^{r_{i}} \textrm{ for}\ i=0,1,\ldots ,l-1, \\
&&\varphi(q)=y^{r} 
\end{eqnarray*}
where $(r_{0},\ldots ,r_{l-1},r)\in \mathbb{Z}^{l+1}$. 
Let also $\theta: K[y,y^{-1}] \rightarrow K(\eta), y \mapsto \eta$ be a specialisation of $K[y,y^{-1}]$ such that $\eta \in \C^*$ and $\theta(\eta_l)$ is a primitive $l$-th root of unity (for simplicity, we may assume that $\theta$ is a $K$-algebra morphism, whence $\theta(\eta_l)=\eta_l$). Let $\Phi$ denote the minimal polynomial of $\eta$ over the field $K$. It follows from Equations \eqref{YokonumaSchur} and \eqref{isom thm} that, for $\ulambda, \umu \in \Pi^{dl}(n)$,
\begin{itemize}
\item the $\Phi$-defect of $V^{\ulambda}$ is equal to $\sum_{i=1}^d \nu_\Phi(\varphi(s_{\ulambda[i]}({\bf q})))$, and \smallbreak
\item $V^{\ulambda}$ and $V^{\umu}$ are in the same $\theta$-block if and only  $V^{\ulambda[i]}$ and $V^{\umu[i]}$ are in the same $\theta$-block of $K(\eta)(\mathcal{H}^{\bf q}_{n_i})_\varphi$, where $n_i=|\ulambda[i]|=|\umu[i]|$, for all $i=1,\ldots,d$.
\end{itemize}
Given the above,  the following result is a direct consequence of Theorem \ref{main-AK}.

\begin{Th}\label{thm above}
Let $\ulambda,\umu \in \Pi^{dl}(n)$. If $V^{\ulambda}$ and $V^{\umu}$ are in the same $\theta$-block, then they have the same $\Phi$-defect.
\end{Th}

Therefore, the analogue of Conjecture \ref{our conj} is  also true for the cyclotomic Yokonuma--Hecke algebra  ${\rm Y}(d,l,n)$, and in particular for the Yokonuma--Hecke algebra  ${\rm Y}(d,1,n)$ of type $A$.

\subsection{Other essential algebras}\label{perspectives}

We have stated Conjecture \ref{our conj} for all cyclotomic Hecke algebras associated with complex reflection groups and we have seen that it holds in all cases for which information on blocks and decomposition matrices is known, that is, for the groups of the infinite series $G(l,p,n)$ and for certain exceptional groups. In the previous subsection we saw that its analogue holds for the Yokonuma--Hecke algebra of type $A$ and for the cyclotomic Yokonuma--Hecke algebra, but one could argue that this is because of the connection, established by the isomorphism \eqref{isom thm}, with the Iwahori--Hecke algebra of type $A$ and the Ariki--Koike algebra respectively. However, one could also wonder whether a similar result holds for all one-parameter essential algebras (including the ones obtained as specialisations of multi-parameter ones as in the end of \S \ref{sec-essential}). One could go a step further and wonder whether a similar result holds for any essential algebra, with $\Phi$-defect being replaced by $ \Psi_{V,i}(M_{V,i})$-defect, using the notation of Definition  \ref{def-essential}. Unfortunately, we do not have any data or examples outside the ones treated in this paper to back up this claim. Proving the claim on a generic level would help a lot though in avoiding to do a huge case-by-case analysis in order to prove Conjecture \ref{our conj} for the exceptional complex reflection groups.

Another perspective to consider is the case of positive characteristic. Throughout this paper we only consider specialisations to subfields of $\C$. Moreover, one can use classical modular representation theory of finite groupes to see that not all simple modules inside a $p$-block (where $p$ is a prime number) have the same defect. However, if we consider a cyclotomic Ariki--Koike algebra as in    
Remark \ref{good Ariki-Koike} and a specialisation $\theta: \Z_K[y,y^{-1}] \rightarrow k, y \mapsto \eta$, where $k$ is any field and $\eta \in k$ is a root of unity of order $e>1$, then Theorem \ref{main-AK} is still valid with a proof similar to the one we provided if the characteristic of $k$ is large enough (to be precise, larger than $e+n-1$). It is also valid if $\eta=1$ and $k$ is a field of characteristic $e>0$. These statements can, through the Morita equivalence of Dipper and Mathas, generalise to any cyclotomic Ariki--Koike algebra. We decided not to include these results for the sake of uniformity, and in order to restrict ourselves to the proof of Conjecture \ref{our conj} in its current form.
However, we do not exclude the possibility that other Hecke algebras and, more generally, other essential algebras might share similar properties.


\begin{thebibliography}{99}    


\addcontentsline{toc}{chapter}{Bibliography}

\bibitem[Ar1]{Arsem}  \textsc{S.~Ariki}, On the semi-simplicity of the Hecke algebra of 
$(\mathbb{Z}/r\mathbb{Z}) \wr \mathfrak{S}_n$. J. Algebra 169 (1994), 216--225.

\bibitem[Ar2]{Ar} \textsc{S.~Ariki},  {Representation theory of a Hecke algebra of 
$G(r, p, n)$}. J.~Algebra {177} (1995), 164--185.
%
%
\bibitem[ArKo]{ArKo} \textsc{S.~Ariki and  K.~Koike}, {A Hecke algebra of $(\mathbb{Z}/r\mathbb{Z})\wr S_n$ and construction of its irreducible representations}. Adv.~Math.~\textbf{106} (1994), 216--243.
%
\bibitem[Ben]{Ben} \textsc{M.~Benard}, {Schur indices and splitting fields of the unitary reflection
groups}. J.~Algebra {38} (1976), 318--342.
%
\bibitem[BGO]{BGO}
\textsc{C. Bessenrodt, J.-B. Gramain, and  J. Olsson},
Generalized hook lengths in symbols and partitions. J. Algebraic Combinatorics 36 (2011), 309--332.
%
\bibitem[Bes1]{Bes1} \textsc{D.~Bessis}, {Sur le corps de d{\'e}finition d'un
groupe de r{\'e}flexions complexe}. Comm.~ Algebra {25}(8) (1997),
2703--2716.
%
\bibitem[Bes2]{Bes2} \textsc{D.~Bessis}, {Finite complex reflection arrangements are $K(\pi, 1)$}. Ann.~Math.~{181}(3) (2015), 809--904.

\bibitem[BCC]{BCC} \textsc{C.~Boura, E.~Chavli and M.~Chlouveraki},  {The BMM symmetrising trace conjecture for the exceptional 2-reflection groups of rank 2}. J.~Algebra {558} (2020), 176--198.

%
\bibitem[BCCK]{BCCK} \textsc{C.~Boura, E.~Chavli, M.~Chlouveraki and  K.~Karvounis},  {The BMM symmetrising trace conjecture for groups $G_4$, $G_5$, $G_6$, $G_7$, $G_8$}. J.~ Symbolic Comput.~ {96} (2020), 62--84.
%
\bibitem[Bou]{Bou05} \textsc{N.~Bourbaki}, Lie groups and Lie algebras. Chapters 4--6, Elements of mathematics, English translation of ``Groupes et alg\`ebres de Lie'', Springer, 2005.
%
%
%
\bibitem[BreMa]{BreMa} \textsc{K.~Bremke and  G.~Malle,} {Reduced words and a length function for $G(e, 1, n)$}. Indag. Math. {8} (1997), 453--469.
%
%
\bibitem[BroMa]{BM} \textsc{M.~Brou\'e and  G.~Malle,} {Zyklotomische Heckealgebren}.  Ast\'erisque  {212}  (1993), 119--189.
%
\bibitem[BMM]{BMM}  \textsc{M.~Brou{\'e}, G.~Malle and  J.~Michel,} {Towards Spetses I}. Trans.~Groups {4} (1999), 157--218.
%
%
%
\bibitem[BMR]{BMR}  \textsc{M.~Brou{\'e}, G.~Malle and R.~Rouquier}, {Complex 
	reflection groups, braid groups, Hecke algebras}. J.~reine angew.
Math. {500} (1998), 127--190.
%
%
\bibitem[BrKl]{BrKl} \textsc{J.~Brundan and A.~Kleshchev}.
Blocks  of  cyclotomic  Hecke  algebras  and  Khovanov--Lauda algebras. Invent. Math.178 (2009), 451--484.
%
\bibitem[Cha1]{Ch18}  \textsc{E.~Chavli}, {Universal  deformations  of  the  finite  quotients  of  the  braid  group  on  $3$  strands}. J.~Algebra {459} (2016), 238--271.
%
\bibitem[Cha2]{Ch17}  \textsc{E.~Chavli}, {The BMR freeness conjecture for the tetrahedral and octahedral family}. Comm.~Algebra {46}(1) (2018), 386--464.
%
\bibitem[Cha3]{Ch20}  \textsc{E.~Chavli},  Decomposition matrices for the generic Hecke algebras on $3$ strands in characteristic zero. Algebras and Representation Theory 23 (2020), 1001--1030. 



\bibitem[Chl1]{C}
\textsc{M.Chlouveraki},
Blocks and families for cyclotomic Hecke algebras. Lecture Notes in Mathematics 1981, Springer-Verlag Berlin Heidelberg, 2009,

\bibitem[Chl2]{HDR}
\textsc{M.~Chlouveraki}, Hecke algebras, generalisations and representation theory. Habilitation thesis, available at:
https://tel.archives-ouvertes.fr/tel-01411063


\bibitem[ChJa1]{CJG}
\textsc{M.~Chlouveraki and N.~Jacon}, 
Schur elements and basic sets for cyclotomic Hecke algebras. 
J. Algebra and its Applications 10 (2011), no. 5, 979--993.


\bibitem[ChJa2]{CJ}
\textsc{M.~Chlouveraki and N.~Jacon},
Schur elements for Ariki--Koike algebras and applications. 
J. Algebraic Combinatorics, Volume 35, N. 2 (2012), 291--311.

\bibitem[ChMi]{CM}
\textsc{M.~Chlouveraki and H.~Miyachi},
Decomposition matrices for $d$-Harish-Chandra series : the exceptional rank 2 cases. 
LMS Journal of Computation and Mathematics 14 (2011), 271--290.

\bibitem[ChPo]{CP}
\textsc{M.~Chlouveraki and G.~Pouchin},
Representation theory and an isomorphism theorem for the framisation of the Temperley--Lieb algebra.  Math. Z. 285, No. 3 (2017), 1357--1380.

\bibitem[CPA1]{CPA}
 \textsc{M. Chlouveraki and L. Poulain d'Andecy},
  Representation theory of the Yokonuma--Hecke algebra. Adv. Math. 259 (2014), 134-172.
  
  
  \bibitem[CPA2]{CPA2}
 \textsc{M. Chlouveraki and L. Poulain d'Andecy},
  Markov traces on affine and cyclotomic Yokonuma--Hecke algebras. Int. Math. Res. Not., Vol. 2016, No. 14 (2016), 4167--4228.

  \bibitem[CuRe]{CuRe} 
  \textsc{C.~W.~Curtis and I.~Reiner}, Representation theory of finite groups and associative algebras. Wiley, New York, 1962; reprinted 1988 as Wiley Classics Library Edition.

\bibitem[DiMa]{DiMa}
\textsc{R. Dipper and A. Mathas},
 Morita equivalences of Ariki--Koike algebras. 
 Math. Z. 240 (2002), no. 3, 579--610.
 
 
 \bibitem[EsRy]{EsRy}  \textsc{J. Espinoza and S. Ryom-Hansen}, Cell structures for the Yokonuma--Hecke algebra and the algebra of braids and ties. J. Pure and Applied Algebra, Volume 222, Issue 11(2018), 3675--3720.
 
 \bibitem[Fa]{F}
\textsc{M. Fayers},
Weights of multipartitions and representations of Ariki--Koike algebras. Adv. Math. 206 (2006) 112-144.

\bibitem[Ge1]{G90}
\textsc{M. Geck}, On the classification of $l$-blocks of finite groups of Lie type. J.~Algebra 151 (1992),180--191.

\bibitem[Ge2]{G}
\textsc{M. Geck},
Brauer trees of Hecke algebras. Comm. Algebra 20 (1992), 2937--2973


\bibitem[GHLMP]{chevie2} \textsc{M.~Geck, G.~Hiss, F.~L\"ubeck, G.~Malle, and G.~Pfeiffer}, {CHEVIE -- a system for computing and processing generic character tables}.
    Applicable Algebra in Engineering Comm. and Computing {7} (1996), 175--210.


\bibitem[GIM]{GIM} \textsc{M. Geck, L. Iancu, G. Malle}, Weights of Markov traces and generic degrees. Indag. Math. (N.S.)
11(3) (2000), 379--397.

\bibitem[GecJa]{GJ}
\textsc{M. Geck and  N. Jacon},
Irreducible Representations of Hecke algebras at roots of unity. Algebras and Applications,  Springer-Verlag London ltd ( 2011). 


\bibitem[GePf]{GP}
\textsc{M. Geck and  G. Pfeiffer},
 Characters of finite Coxeter groups and Iwahori--Hecke algebras. London Mathematical Society Monographs. New Series, 21. The Clarendon Press, Oxford University Press, New York, 2000. xvi+446 pp. 
 
 \bibitem[GeRo]{GR}
 \textsc{M. Geck and R. Rouquier}, 
 Centers and simple modules for Iwahori-Hecke algebras. In: Finite reductive groups, related structures and representations (ed. M. Cabanes), Progress in Math. 141, Birkh\"auser (1997), 251--272.
 
 
\bibitem[GenJa]{GJa}
\textsc{G. Genet and N. Jacon,}
Modular representations of cyclotomic Hecke algebras of type $G(r,p,n)$.
Int. Math. Res. Not. 2006, Art. ID 93049, 18 pp.



\bibitem[Ja]{JKle}
\textsc{N. Jacon},
Kleshchev multipartitions and extended Young diagrams.
Advances in Mathematics 339 (2018)  367--403.

\bibitem[JaLe]{JL}
\textsc{N. Jacon and C.  Lecouvey},
Cores of Ariki-Koike algebras. Doc. Math. 26  (2021), 103--124.

\bibitem[JPA]{JPA}
\textsc{N. Jacon and L. Poulain d'Andecy},
 An isomorphism theorem for Yokonuma--Hecke algebras and applications to link invariants. Math. Z. 283 (2016), no. 1, 301--338
 
 \bibitem[Jam]{Jam}
 \textsc{G.~D.~James}, The decomposition matrices of ${\rm GL}_n(q)$ for $n \leq10$, Proc. London Math. Soc. 60 (1990), 225--265.
 
 \bibitem[Ju1]{Ju1} \textsc{J. Juyumaya}, Sur les nouveaux g\'en\'erateurs de l'alg\`ebre de Hecke $H(G,U,1)$. J. Algebra 204 (1998) 49--68.
 
\bibitem[Ju2]{Ju2} \textsc{J. Juyumaya}, Markov trace on the Yokonuma--Hecke algebra. J. Knot Theory Ramifications 13 (2004) 25--39.

\bibitem[JuKa]{JuKa} \textsc{J. Juyumaya and S. Kannan}, Braid relations in the Yokonuma--Hecke algebra. J. Algebra 239 (2001) 272--297.


\bibitem[Lu]{Lus} \textsc{G. Lusztig,} Character sheaves on disconnected groups VII. Represent. Theory 9 (2005), 209--266.


\bibitem[Mal1]{Ma1} \textsc{G.~Malle}, Unipotente Grade imprimitiver komplexer Spiegelungsgruppen. J.
Algebra 177 (1995), 768--826.

\bibitem[Mal2]{Ma2}   \textsc{G.~Malle}, {Degr{\'e}s relatifs des alg{\`e}bres cyclotomiques associ{\'e}es aux groupes de r{\'e}flexions complexes de dimension deux}. Progress in
Math. {141}, Birkh{\"a}user (1996), 311--332.


\bibitem[Mal3]{Ma4}   \textsc{G.~Malle}, {On the rationality and fake degrees of characters of cyclotomic algebras}. J.~Math.~Sci.~Univ.
Tokyo {6 }(1999), 647--677.

\bibitem[Mal4]{Ma5}   \textsc{G.~Malle}, {On the generic degrees of cyclotomic algebras}. Represent. Theory {4 }(2000), 342--369.



\bibitem[MalMat]{MaMa}  \textsc{G.~Malle and  A.~Mathas}, {Symmetric cyclotomic Hecke algebras}. J. Algebra {205}(1) (1998), 275--293.


\bibitem[MalMi]{MM10}  \textsc{G.~Malle and  J.~Michel,} {Constructing representations of Hecke algebras for complex
reflection groups}. LMS J. Comput. Math. {13} (2010), 426--450.

\bibitem[MaNa]{MN}  \textsc{G.~Malle and  G.Navarro,} {Blocks with equal height zero degrees.} Trans. Amer. Math. Soc. 363
(2011), no. 12, 6647--6669. 

\bibitem[Mar1]{Mar41}   \textsc{I.~Marin}, {The cubic Hecke algebra on at most $5$ strands}. J. Pure Appl. Algebra {216} (2012), 2754--2782.


 \bibitem[Mar2]{Mar43}  \textsc{ I.~Marin,} {The freeness conjecture for Hecke algebras of complex reflection groups, and the case of the Hessian group $G_{26}$}. J. Pure Appl. Algebra {218} (2014), 704--720.


\bibitem[Mar3]{MarNew}   \textsc{I.~Marin,} {Proof of the BMR conjecture for $G_{20}$ and $G_{21}$}. J. Symbolic Computation {92} (2019), 1--14.


\bibitem[MarPf]{MaPf} \textsc{ I.~Marin and G.~Pfeiffer,} {The  BMR  freeness  conjecture  for  the  $2$-reflection  groups}.  Math.~Comput. {86} (2017), 2005--2023.


\bibitem[MarWa]{Mar46} \textsc{ I.~Marin and E.~Wagner,} {Markov traces on the BMW algebras}. arXiv:1403.4021.

\bibitem[Mat]{Mat} \textsc{A. Mathas}, Matrix units and generic degrees for the Ariki-Koike algebras. J. Algebra 281 (2004),
695--730.

\bibitem[Mi]{chevie1} \textsc{J.~Michel}, {The development version of the CHEVIE package of GAP3}. J.~Algebra {435} (2015), 308--336.



\bibitem[PdA]{PdA} \textsc{L.~Poulain d'Andecy,} {Invariants for links from classical and affine Yokonuma--Hecke algebras}. In: Algebraic Modeling of Topological and Computational Structures and Applications, Springer Proceedings in Mathematics \& Statistics 219 (2017), 77--95.

\bibitem[Ros]{Salim} 
\textsc{S. Rostam}, Cyclotomic quiver Hecke algebras and Hecke algebra of $G(r,p,n)$. Transactions of the AMS 371 no. 6 (2019) 3877--3916.


\bibitem[ShTo]{ShTo} \textsc{G.~C.~Shephard and  J.~A.~Todd,} {Finite unitary reflection
	groups}. Canad.~J.~Math. {6} (1954), 274--304.
	
\bibitem[Th1]{thi} \textsc{N.~Thiem}, {Unipotent Hecke algebras: the structure, representation theory, and combinatorics}. Ph.D. Thesis, University of Wisconsin (2004).

\bibitem[Th2]{thi2} \textsc{N.~Thiem}, {Unipotent Hecke algebras of \,${\rm GL}_n(\mathbb{F}_q)$}. J.~Algebra {284} (2005) 559--577. 

\bibitem[Th3]{thi3} \textsc{N.~Thiem}, {A skein-like multiplication algorithm for unipotent Hecke algebras}. Trans.~Amer.~Math.~Soc.~{359}(4) (2007) 1685--1724.

\bibitem[Tsu]{Tsu} \textsc{S.~Tsuchioka}, {BMR freeness for icosahedral family}. Experimental Mathematics (2018), doi: 10.1080/10586458.2018.1455072.


\bibitem[Wa]{W}
\textsc{K. Wada},  Blocks of category $\mathcal{O}$ for rational Cherednik algebras and of cyclotomic Hecke algebras of type $G(r,p,n)$.
Osaka J. Math. 48 (2011), 895-9

\bibitem[Web]{Web} Webpage of the GAP3 program: \emph{http://chlouveraki.perso.math.cnrs.fr/gap3}

\bibitem[Wi]{Wi} \textsc{G.~Williamson}, Schubert calculus and torsion explosion, J. Amer. Math. Soc. 30 (2017), 1023--1046 

\bibitem[Yo]{yo} \textsc{T.~Yokonuma}, {Sur la structure des anneaux de Hecke d'un groupe de Chevalley fini}. C.~R.~Acad.~Sci.~Paris Ser.~I Math.~{264}  (1967)  344--347.

\end{thebibliography}
\end{document}